\documentclass[12pt]{article}

\usepackage{epsfig,amsmath,amssymb,graphics,verbatim,amsfonts,psfrag,amsthm}
\usepackage[letterpaper,margin=0.85in]{geometry}
\usepackage{paralist}
\usepackage{xcolor}
\usepackage[font={small,it}]{caption}
\usepackage{subcaption}
\usepackage{graphicx}
\usepackage{epic}
\usepackage[Option]{overpic}
\graphicspath{{figures/}}
  
\newtheorem{thm}{Theorem}[section]

\newtheorem{prop}[thm]{Proposition}

\newtheorem{ex}[thm]{Example}

\newtheorem{remark}[thm]{Remark}
\newtheorem*{theorem-nonA}{Theorem A}

\newtheorem{defn}[thm]{Definition}

\DeclareMathOperator{\trace}{trace}

\DeclareMathOperator{\id}{id}

\DeclareMathOperator{\supp}{supp}

\def\txtd{{\textnormal{d}}}
\def\txte{{\textnormal{e}}}

\def\txtf{{\textnormal{f}}}
\def\txtp{{\textnormal{p}}}
\def\txts{{\textnormal{s}}}

\def\txtD{{\textnormal{D}}}
\def\txtT{{\textnormal{T}}}

\def\R{\mathbb{R}}

\def\E{\mathbb{E}}
\def\I{\infty}
\newcommand{\rmd}{\mathrm{d}}
\newcommand{\rmD}{\mathrm{D}}
\renewcommand{\epsilon}{\varepsilon}

\newcommand{\fa}          {\quad \text{for all } \,}
\newcommand{\faa}          {\quad \text{for almost all } \,}

\newcommand{\be}{\begin{equation}}
\newcommand{\ee}{\end{equation}}
\newcommand{\bea}{\begin{eqnarray}}
\newcommand{\eea}{\end{eqnarray}}
\newcommand{\beann}{\begin{eqnarray*}}
\newcommand{\eeann}{\end{eqnarray*}}
\newcommand{\benn}{\begin{equation*}}
\newcommand{\eenn}{\end{equation*}}

\def\ra{\rightarrow}
\def\I{\infty}

\newcommand{\cK}{{\mathcal K}}  
\newcommand{\cM}{{\mathcal M}}  
\newcommand{\cN}{{\mathcal N}}  
\newcommand{\cU}{{\mathcal U}}  
\newcommand{\cX}{{\mathcal X}}  

\numberwithin{equation}{section}
\begin{document}

\author{Maximilian Engel\thanks{Department of Mathematics, Freie Universit{\"a}t Berlin, Arnimallee 6, 14195 Berlin, Germany.}~~and Christian Kuehn\thanks{Faculty of Mathematics, Technical University of Munich,
Boltzmannstr. 3, D-85748 Garching bei M\"{u}nchen.}}
 
\title{A random dynamical systems perspective on isochronicity for stochastic oscillations}

\maketitle

\begin{abstract}
For an attracting periodic orbit (limit cycle) of a deterministic dynamical system, one defines the isochron for each point of the orbit as the cross-section with fixed return time under the flow. Equivalently, isochrons can be characterized as stable manifolds foliating neighborhoods of the limit cycle or as level sets of an isochron map. In recent years, there has been a lively discussion in the mathematical physics community on how to define isochrons for stochastic oscillations, i.e.~limit cycles or heteroclinic cycles exposed to stochastic noise. 
The main discussion has concerned an approach finding stochastic isochrons as sections of equal expected return times versus the idea of considering eigenfunctions of the backward Kolmogorov operator. 
We discuss the problem in the framework of random dynamical systems and introduce a new rigorous definition of stochastic isochrons as random stable manifolds for random periodic solutions with noise-dependent period. This allows us to establish a random version of isochron maps whose level sets coincide with the random stable manifolds. Finally, we discuss links between the random dynamical systems interpretation and the equal expected return time approach via averaged quantities.
\end{abstract}

 \bigskip



\section{Introduction}
\label{sec:intro}

Periodic behavior is ubiquitous in the natural sciences and in engineering. Accordingly, many mathematical models of dynamical systems, usually given by ordinary differential equations (ODEs), are characterized by the existence of attracting periodic orbits, also called \emph{limit cycles}. Interpreting the limit cycle as a ``clock'' for the system, one can ask which parts of the state space can be associated with which ``time'' on the clock.

It turns out that one can generally divide the state space into sections, called \emph{isochrons}, intersecting the asymptotically stable periodic orbit. Trajectories starting on a particular isochron all converge to the trajectory starting at the intersection of the isochron and the limit cycle. Hence, each point in the basin of attraction of the limit cycle can be allocated a time on the periodic orbit, by belonging to a particular isochron. Isochrons can then be characterized as the sections intersecting the limit cycle, such that the return time under the flow to the same section always equals the period of the attracting orbit and, hence, the return time is the same for all isochrons. The analysis of ODEs provides additional characterizations of isochrons, involving, for example, an isochron map or eigenfunctions of associated operators.

Clearly, mathematical models are simplifications which often leave out parameters and details of the described physical or biological system. Hence, a large number of degrees of freedom is inherent in the modeling. The introduction of random noise is often a suitable way to integrate such non-specified components into the model such that, for example, an ODE becomes a stochastic differential equation (SDE). Examples for stochastic oscillators/oscillations can be found in a wide variety of applications such as neuroscience~\cite{Bauermeisteretal,BrooksBressloff,Lindneretal,SuRubinTerman}, ecology~\cite{SchreiberBenaimAtchade,SadhuKuehn}, bio-mechanics~\cite{GatesSuDingwell,RevzenGuckenheimer}, geoscience~\cite{BenziParisiSuteraVulpiani,NicolisNicolis}, among many others. In addition, stochastic oscillations have become a recently very active research topic in the rigorous theory of stochastic dynamical systems with small noise~\cite{BaudelBerglund,BerglundGentzKuehn1,BerglundLandon,Giacominetal2018}. 

Lately, there has been a lively discussion~\cite{Pikovskycomment,ThomasLindnerReply} in the mathematical physics community about how to extend the definition and analysis of isochrons to the stochastic setting. As pointed out above, there are several different characterizations in the deterministic case inspiring analogous stochastic approaches. So far, there are two main approaches to define stochastic isochrons in the physics literature, both focused on stochastic differential equations. One approach, due to Thomas and Lindner \cite{ThomasLindner}, focuses on eigenfunctions of the associated infinitesimal generator $\mathcal L$. The other one is due to Schwabedal and Pikovsky \cite{SchwabedalPikovsky}, who introduce isochrons for noisy systems as sections $W^\E(x)$ with the mean first return time to the same section $W^\E(x)$ being a constant $\bar T$, equaling the average oscillation period. Cao, Lindner and Thomas \cite{caolindnerthomas19} have used the Andronov-Vitt-Pontryagin formula, involving the \emph{backward Kolmogorov operator} $\mathcal L$, with appropriate boundary conditions to establish the isochron functions for $W^\E(x)$ more rigorously. 

These approaches have in common that they focus on the ``macroscopic'' or ``coarse-grained'' level by considering averaged objects and associated operators. We complement the existing suggestions by a new approach within the theory of \emph{random dynamical systems} (see e.g.~\cite{a98}) which has proven to give a framework for translating many deterministic dynamical concepts into the stochastic context. A random dynamical system in this sense consists of a model of the time-dependent noise formalized as a a dynamical system $\theta$ on the probability space, and a model of the dynamics on the state space formalized as a \textit{cocycle} $\varphi$ over $\theta$.
This point of view considers the asymptotic behaviour of typical trajectories. As trajectories of random dynamical systems depend on the noise realization, any convergent behaviour of individual trajectories to a fixed attractor cannot be expected. The forward in time evolution of sets under the same noise realization yields the \emph{random forward attractor} $A$ which is a time-dependent object with fibers $A(\theta_t \omega)$.
An alternative view point is to consider, for a fixed noise realization $\omega \in \Omega$, the flow of a set of initial conditions from time $t=-T$ to a fixed endpoint in time, say $t=0$, and then take the (pullback) limit $T\to\infty$. If trajectories of initial conditions converge under this procedure to fibers $\tilde A(\omega)$ of some random set $\tilde A$, then this set is called a \emph{random pullback attractor}.

In this paper, we will consider mainly situations where the random dynamical system is induced by an SDE and there exists a random (forward and/or pullback) attractor $A$ which is topologically equivalent to a cycle for each noise realization, i.e.~a \emph{attracting random cycle}, whose existence can be made generic in a suitably localized setup around a deterministic limit cycle. We will extend the definition of a random periodic solution $\psi$ \cite{ZhaoZheng09} living on such a random attractor to situations where the period is random, giving a pair $(\psi, T)$. Isochrons can then be defined as \emph{random stable manifolds} $W^{\txtf}(\omega, x)$ for points $x$ on the attracting random cycle $A(\omega)$, in particular for random periodic solutions. We usually consider situations with a spectrum of exponential asymptotic growth rates, the \textit{Lyapunov exponents} $\lambda_1 > \lambda_2 > \dots > \lambda_p$, which allows to transform the idea of hyperbolicity to the random context. Additionally, we can introduce a time-dependent \emph{random isochron map} $\tilde \phi$, such that the isochrons are level sets of such a map. Hence, on a pathwise level, we achieve a complete generalization of deterministic to random isochronicity, which is the key contribution of this work.
The main results can be summarized in the following theorem:
\begin{theorem-nonA}
Assume the random dynamical system $(\theta, \varphi)$ on $\mathbb{R}^m$ has a hyperbolic random limit cycle, supporting a random periodic solution with possibly noise-dependent period. Then, under approptriate assumptions on smoothness and boundedness,
\begin{enumerate} 
\item the random forward isochrons are smooth invariant random manifolds which foliate the stable neighbourhood of the random limit cycle on each noise fibre,
\item there exists a smooth and measurable (non-autonomus) random isochron map $\tilde \phi$ whose level sets are the random isochrons and whose time derivative along the random flow is constant.
\end{enumerate}
\end{theorem-nonA}

The remainder of the paper is structured as follows. Section~\ref{sec:deterministiccase} gives an introduction to the deterministic theory of isochrons, summarizing the main properties that we can then transform into the random dynamical systems setting. The latter is discussed in Section~\ref{sec:RDS}, where we elucidate the notions of Lyapunov exponents, random attractors and, specifically, random limit cycles and their existence. 
Section~\ref{sec:random_isochrons} establishes the two main statements, contained in Theorem A: in Section~\ref{sec:stable_manifold}, we show Theorem~\ref{thm:SMT}, summarizing different scenarios in which random isochrons are random stable manifolds foliating the neighbourhoods of the random limit cycle. 
In Section~\ref{sec:randomisomap}, we prove Theorem~\ref{def:randomisochronmap}, generalizing characteristic properties of the isochron map to the random case. 
We conclude Section~\ref{sec:random_isochrons} with an elaboration on the relationship between expected quantities of the RDS approach and the definition of stochastic isochrons via mean first return times, i.e.,~one of the main physics approaches. 
Additionally, the paper contains a brief conclusion with outlook, and an appendix with some background on random dynamical systems.

\section{The deterministic case} \label{sec:deterministiccase}

The basic facts about isochrons have been established in~\cite{Guckenheimer13}. Here we 
summarize some facts restricted to the state space $\cX = \R^m$ but the theory easily lifts to ordinary differential equations (ODEs) on smooth manifolds $\cM=\cX$. Consider an ODE
\be
\label{eq:ODE}
x'=f(x),\qquad x(0)=x_0 \in\R^m,
\ee
where $f$ is $C^k$ for $k\geq 1$. Let $\Phi(x_0,t)=x(t)$ be the flow 
associated to~\eqref{eq:ODE} and suppose $\gamma=\{\gamma(t)\}_{t\in[0,\tau_\gamma]}$ is a 
hyperbolic periodic orbit with minimal period $\tau_\gamma>0$. A \emph{cross-section} $\cN\subset 
\R^m$ at $x\in\gamma$ is a submanifold such that $x\in \cN$, $\bar{\cN}\cap \gamma=\{x\}$, and
\benn
\txtT_x\cN \oplus \txtT_x \gamma = \txtT_x\R^m\simeq \R^m,
\eenn 
i.e.~the submanifold $\cN$ and the orbit $\gamma$ intersect transversally.

Let $g:\cN\ra \cN$ be the \emph{Poincar\'e map} defined by the first return of $y\in\cN$ under
the flow $\Phi$ with $\cN$ (see Figure~\ref{fig:det_isochron_sketch}); locally near any point $x\in \gamma$ the map $g$ 
is well-defined. For simplicity (and with the look forward towards the noisy case) let us
assume that $\gamma$ is a \emph{stable hyperbolic periodic orbit}, i.e.~the eigenvalues $\mu_i$ of $\txtD g(x)$, also called \emph{characteristic multipliers}, satisfy $\mu_1 = 1$ and $\left|\mu_2\right|, \dots, \left|\mu_m\right| < 1$, counting multiplicities. The numbers 
$$\lambda_i = \frac{1}{T}\ln \mu_i$$
are called the \emph{characteristic exponents}
(for more background on the stability of linear non-autonomous systems and associated \emph{Floquet theory} see e.g~\cite[Chapter 2.4]{Chicone}). We call such a stable hyperbolic periodic orbit a \emph{stable (hyperbolic) limit cycle} since there is a neighbourhood $\cU$ of $\gamma$ such that for $y \in \cU$ we have $\txtd(\Phi(y,t),\gamma) \to 0$, as $t \to \infty$, where $\txtd$ is the Euclidean metric on $\R^m$.
In particular, note that there is a lower bound on the speed of exponential convergence to the limit cycle, given by
$$
\lambda := \min_{i: \lambda_i \neq 0} \Re (-\lambda_i) > 0.
$$
We give a definition of \emph{isochrons} as stable sets and then establish its equivalence to level sets of a specific map. We further find these level sets to be cross-sections to $\gamma$ for which the time of first return is identical to the period $\tau_{\gamma}$, explaining the name isochrons.
\begin{defn}
The \emph{isochron} $W(x)$ of a point on a hyperbolic limit cycle $x\in\gamma$ is given by its stable set
\be
\label{eq:isochron}
W(x) :=\left\{y\in\R^m:\lim_{t\ra +\I}\txtd(\Phi(x,t),\Phi(y,t))=0\right\}.
\ee  
In particular, due to hyperbolicity, we have  that for every $ \tilde \lambda \in (0, \lambda)$
\be
\label{eq:isochronhyp}
W(x) =\left\{y\in\R^m:\sup_{t \geq 0}\txte^{\tilde \lambda t} \, \txtd(\Phi(x,t),\Phi(y,t)) < \infty \right\}.
\ee  
\end{defn}
It is by now classical that stable sets are manifolds and for each 
$x\in \gamma$, we get a stable manifold $W^\txts(x)$ diffeomorphic to $\R^{m-1}$, precisely coinciding with the isochron $W(x)$.
We can foliate a neighbourhood
$\cU$ of $\gamma$ by the manifolds $W(x)$ and these manifolds are permuted 
by the flow since 
\begin{equation} \label{eq:permutedflow}
W(\Phi(x,t))=\Phi(W(x),t)\quad \forall t\in\R.
\end{equation}
We summarize these crucial observations in the following theorem.
\begin{thm}[Theorem A in \cite{Guckenheimer13}, Theorem 2.1 in \cite{Giacominetal2018}] \label{thm:isochonsstableman_det}
Consider the flow $\Phi: \mathbb{R}^m \times \mathbb{R} \to \mathbb{R}^m$ for the ODE~\eqref{eq:ODE} with hyperbolic stable limit cycle $\gamma=\{\gamma(t)\}_{t\in[0,\tau_\gamma]}$. Then the following holds:
\begin{enumerate}
\item For each $x \in \gamma$, the isochron $W(x)$ is an $(m-1)$-dimensional manifold transverse to $\gamma$, in particular it is a cross-section of $\gamma$, of the same regularity as the vector field $f$ in the ODE~\eqref{eq:ODE} (i.e.~$C^k$ if $f$ is $C^k$).
\item The stable manifold $W^\txts(\gamma)$ contains a full neighbourhood of $\gamma$ and can be written as
\benn
W^\txts(\gamma)=\bigcup_{x\in \gamma} W(x),
\eenn 
where the union of isochrons is disjoint.
\item The map $\xi: W^\txts(\gamma) \to \mathbb{R} \mod \tau_{\gamma} $, also called the \emph{isochron map}, is given for every $y \in W^\txts(\gamma)$ as the unique $t$ such that $y \in W(\gamma(t))$, i.e.
\begin{equation} \label{isochronmap}
  \lim_{s\ra +\I}\txtd(\Phi(\gamma(\xi(y)),s),\Phi(y,s))=\lim_{s\ra +\I}\txtd(\gamma(s + \xi(y)),\Phi(y,s))=0 \,,
\end{equation}
and $\xi$ is also $C^k$.
\end{enumerate}
\end{thm}
Using the properties established in Theorem~\ref{thm:isochonsstableman_det},
we can derive the following well-known characterizations of the isochrons $W(x)$, $x \in \gamma$, and of the isochron map $\xi$.
\begin{prop} \label{prop:characterize}
Assume that we are in the situation of Theorem~\ref{thm:isochonsstableman_det}. We have that
\begin{enumerate}
\item for each $x \in \gamma$, the isochron $W(x)$ is precisely the level set of $\xi(x)$, i.e.
\be \label{eq:isochron_alter}
W(x)=\left\{y\in W^\txts(\gamma): \ \xi(y) = \xi(x)\right\}\,,
\ee 
\item  the isochron map $\xi: W^\txts(\gamma) \to \mathbb{R} \mod \tau_{\gamma} $ satisfies
\be \label{eq:invariance_det}
\frac{\txtd}{\txtd t} \xi (\Phi(y,t)) = 1 \ \text{ for all } t \geq 0, \, y \in W^\txts(\gamma)\,,
\ee
\item the isochron $W(x)$ is the cross-section $\cN_x$ at $x$ such that
\be
\Phi(\cN_x,\tau_\gamma)\subseteq \cN_x,
\ee
i.e.~the cross-section on which all starting points return in the same time $\tau_{\gamma}$.
\end{enumerate}
\end{prop}
\begin{proof}
The first statement follows from the fact that $\gamma(\xi(x)) = x$ for all $x \in \gamma$ and equation~\eqref{isochronmap} in combination with the definition of $W(x)$: in more detail, we have $y \in W(x)$ if and only if $\lim_{s \to \infty} d(\Phi(x,s), \Phi(y,s)) = 0$ which is equivalent to $\lim_{s \to \infty} d(\gamma(s + \xi(x)), \Phi(y,s)) = 0$ which holds if and only if $\xi(x) = \xi(y)$.

The second statement can be deduced from the invariance property $\Phi(\cdot, t) W(x) = W(\Phi(x,t))$ for any $x \in \gamma$ since it implies for $y \in W(x)$, i.e.~$\xi(y) = \xi(x)$, that
$$ \xi(\Phi(y,t)) = \xi(\Phi(x,t)) = \xi(x) + t \mod \tau_{\gamma} = \xi(y) + t \mod \tau_{\gamma},$$
which is equivalent to the claim.

The third statement can be easily derived from the fact that for all $y \in W^\txts(\gamma)$
\begin{align*}
\lim_{t\ra +\I}\txtd(\gamma(t + \xi(y)),\Phi(\Phi(y,\tau_{\gamma}),t) &= \lim_{t\ra +\I}\txtd(\gamma(t + \xi(y)),\Phi(y,t+ \tau_{\gamma}))\\
 &=\lim_{s\ra +\I}\txtd(\gamma(s - \tau_{\gamma} + \xi(y)),\Phi(y,s))\\
& = \lim_{s\ra +\I}\txtd(\gamma(s + \xi(y)),\Phi(y,s)) = 0\,.
\end{align*} 
This finishes the proof.
\end{proof}

Summarizing, we can view isochrons $W(x)$ as stable manifolds of points on the limit cycle. The sets $W(x)$
are uniquely defined and have codimension one. They locally foliate neighborhoods of the limit
cycle. They can also be characterized and computed as level sets of a specific isochron map whose total derivative along the flow is equal to $1$, by looking for sections
of fixed return time under the flow. In the course of this article, we will transform all the discussed properties to the random case.

Guckenheimer~\cite{Guckenheimer13} tackles additional questions regarding the boundary of $W^\txts(\gamma)$. These questions concern global properties of isochrons. Since we want to first understand a neighbourhood $\cU$ of $\gamma$ in the stochastic setting, we skip these problems here.
With this in mind, we consider an adjustment of the main planar example in~\cite{Guckenheimer13} which does not involve the boundary of $W^\txts(\gamma)$. The example is simple but illuminating and already contains the main aspects of the difficulties in extending isochronicity to the stochastic context, as we will see later.
\begin{ex} \label{ex:Guck}
Consider the ODE
\be
\label{eq:exGuck}
\begin{array}{lcl}
\vartheta' &=& h(r),\\
r' &=&  r(r_1^2-r^2),
\end{array}
\ee
in polar coordinates $(\vartheta,r)\in [0,2\pi)\times (0,+\I)$, where $r_1 > 0$ is fixed, 
$h(r)\geq K > 0$ for some constant $K$, and $h$ is smooth, such that there is always the periodic orbit
$\gamma=\{r=r_1\}$. 
If $h(r)\equiv 1$, then one easily checks that the isochrons of $\gamma$ are (see Figure~\ref{fig:det_isochron_sketch} (a))
\be \label{isochrons_simple}
W((\vartheta_*,r_*))=\{(\vartheta,r):r\in(0,\infty),\vartheta=\vartheta_*\}.
\ee
However, if we consider $h$ such
that $h'(r_1)\neq 0$, then the isochrons bend into \emph{curves}, instead of 
being ``cut-linear'' rays. Indeed, the periodic orbit has period $\tau_{\gamma} = 2\pi/h(r_1)$
but the return time to the same $\vartheta$-coordinate changes near $\gamma$ (see Figure~\ref{fig:det_isochron_sketch} (b)).  
\end{ex}
\begin{figure}[htbp]
        \centering
        \begin{subfigure}{.4\textwidth}
        \centering
  		\begin{overpic}[width=1.0\textwidth]{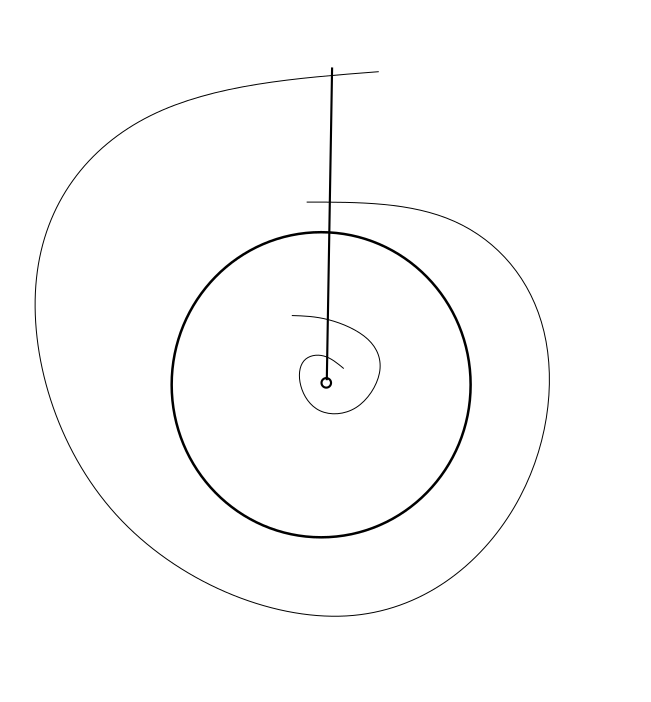}            
                \put(33,31){\small $\gamma$}
                \put(10,65){\footnotesize orbit}
                \put(49,78){\footnotesize isochron}
                \put(49,38){\scriptsize orbit}
        \end{overpic}
        \caption{$h \equiv 1$}
		\end{subfigure}%
		 \begin{subfigure}{.4\textwidth}
        \centering
  		\begin{overpic}[width=1.0\textwidth]{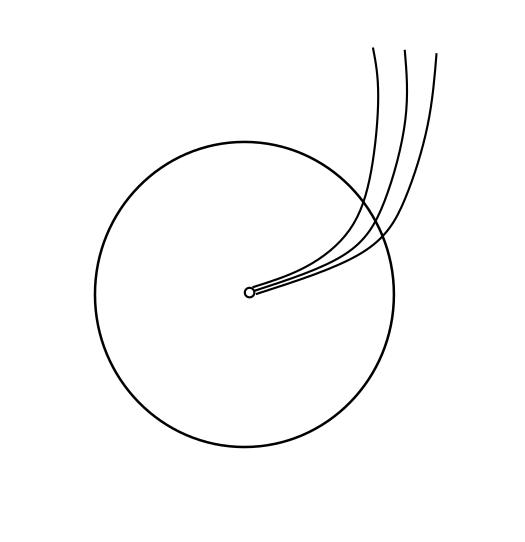}   
  				\put(24,31){\small $\gamma$}
                \put(82,76){\small isochrons}
        \end{overpic}
        \caption{$h'(r_1) \neq 0$}
		\end{subfigure}
\caption{Sketch of isochrons for limit cycle $\gamma$ in example~\eqref{ex:Guck}, with $h \equiv 1$ (a) where the isochrons are simply given by equation~\eqref{isochrons_simple}, and with $h'(r_1) \neq 0$ (b) where the isochrons are curved.}
        \label{fig:det_isochron_sketch}
\end{figure}
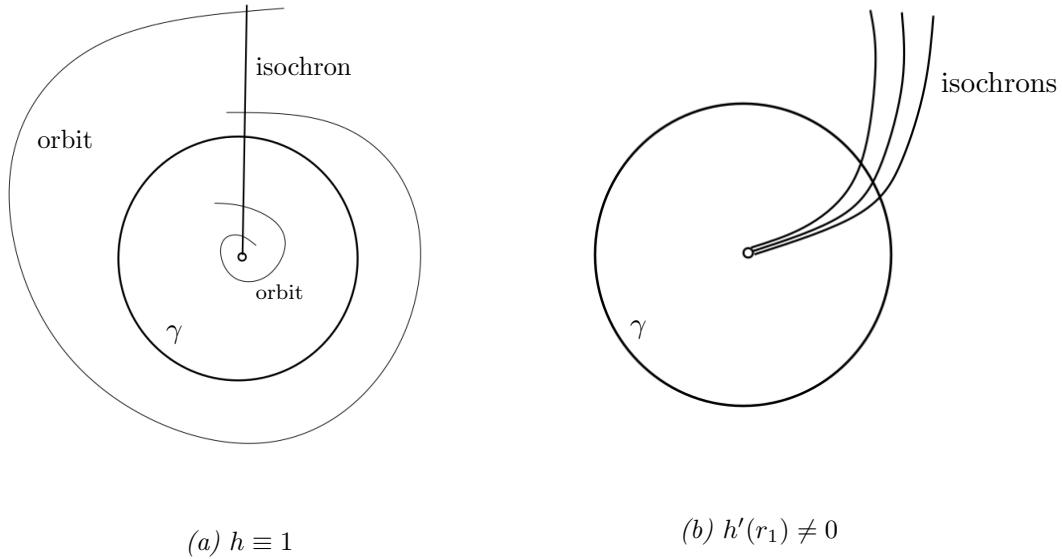
Our considerations indicate that, in order to find isochrons in the stochastic case, 
a first approach is to consider ``stable manifolds'' also for this situation. The most suitable framework for this approach turns out to be the one of random dynamical systems (RDS).


\section{Stochastically driven limit cycles in the framework of random dynamical systems} \label{sec:RDS}
In the following, we develop a theory of isochrons within the framework of random dynamical systems. A continuous-time random dynamical system on a topological state space $\cX$ consists of 
\begin{enumerate}
\item[(i)] a model of the noise on a probability space $(\Omega, \mathcal{F}, \mathbb{P})$, formalized as a measurable flow $(\theta_t)_{t \in \mathbb{R}}$ of $\mathbb P$-preserving transformations $\theta_t: \Omega \to \Omega$,
\item[(ii)] a model of the dynamics on $\cX$ perturbed by noise formalized as a \textit{cocycle} $\varphi$ over $\theta$.
\end{enumerate}
This setting is very helpful to understand properties of dynamical systems under the influence of stochastic noise. In technical detail, the definition of a random dynamical system is given as follows \cite[Definition 1.1.2]{a98}.
\begin{defn}[Random dynamical system] \label{def_RDS}
  Let $(\Omega,\mathcal F,\mathbb P)$ be a probability space and $\cX$ be a topological space. A \emph{random dynamical system (RDS)} is a pair of mappings $(\theta, \varphi)$.
  \begin{itemize}
    \item[$\bullet$] The ($\mathcal B(\R)\otimes \mathcal F$, $\mathcal F$)-measurable mapping $\theta: \R\times \Omega\to \Omega$, $(t,\omega)\mapsto \theta_t\omega$, is a \emph{metric dynamical system}, i.e.
    \begin{itemize}
		\item [(i)]	$\theta_0=\id$ and $\theta_{t+s}=\theta_t\circ \theta_s$ for $t,s\in\mathbb R$,
		\item [(ii)] $\mathbb P (A) = \mathbb P(\theta_t A)$ for all $A\in \mathcal F$ and $t\in\mathbb R$.
	\end{itemize}	
    \item[$\bullet$]
	The ($\mathcal{B}(\mathbb{R})\otimes \mathcal{F} \otimes \mathcal{B}(\cX)$, $\mathcal{B}(\cX)$)-measurable mapping $\varphi: \mathbb{R} \times \Omega \times \cX \to  \cX, (t, \omega, x) \mapsto \varphi(t, \omega, x)$, is a cocycle over $\theta$, i.e.
	\[
	\varphi(0, \omega, \cdot) = \id \quad \text{and} \quad \varphi(t+s, \omega, \cdot) = \varphi(t, \theta_s \omega, \varphi(s, \omega, \cdot))  \quad \text{ for all } \omega \in \Omega \text{ and } t, s \in \mathbb{R}\,.
	\]
  \end{itemize}
The random dynamical system $(\theta,\varphi)$ is called continuous if $(t, x) \mapsto \varphi(t, \omega,x)$ is continuous for every $\omega \in \Omega$. We still speak of a random dynamical system, if its cocycle is only defined in forward time, i.e.~if the mapping $\varphi$ is only defined on $\mathbb{R}_0^+ \times \Omega \times \cX$. We will make it noticeable whenever this is the case.
\end{defn}
In the following, the metric dynamical system $(\theta_t)_{t \in \mathbb{R}}$ is often even ergodic, i.e.~any $A\in\mathcal F$ with $\theta_t^{-1}A =A$ for all $t\in\mathbb R$ satisfies $\mathbb P(A)\in \{0,1\}$.
Note that we define $\theta$ in two-sided time whereas $\varphi$ can be restricted to one-sided time. This is motivated by the fact that a large part of this article will deal with random dynamical systems generated by stochastic differential equations (SDEs).
Hence, we are interested in random dynamical systems adapted to a suitable filtration and of white noise type (see Appendix~\ref{sec:RDSSDE}). In this context, we can understand $\varphi$ as the ``stochastic flow'' induced by solving the corresponding SDE and $\theta_t$ as a time shift on the canonical space $\Omega$ of all continuous paths starting at $0$, equipped with the Wiener measure.

Additionally note that the RDS generates a \emph{skew product flow}, i.e.~a family of maps $(\Theta_t)_{t \in \mathbb{T}}$ from $\Omega \times \cX $ to itself such that for all $t \in \mathbb T$ and $\omega \in \Omega, x \in \cX$
\begin{equation} \label{eq:skewproductflow}
\Theta_t(\omega, x) = (\theta_t \omega, \varphi(t, \omega,x))\,.
\end{equation}
\subsection{Differentiability and Lyapunov exponents}
\label{sec:differentiability}
The random dynamical system $(\theta, \varphi)$ is called $C^k$ if $\varphi(t, \omega, \cdot) \in C^k$ for all $t\in\mathbb T$ and $\omega\in\Omega$, where again $ \mathbb T \in \{ \R, \R_0^+\}$. As in the deterministic case, let us assume that the state space is $\cX = \mathbb{R}^m$ (the following can also be extended to smooth $m$-dimensional manifolds as in Appendix~\ref{LyapSpec}) and that $(\theta, \varphi)$ is  $C^1$. The \textit{linearization} or \textit{derivative} $\rmD \varphi(t,\omega,x)$ of $\varphi(t,\omega,\cdot)$ at $x \in \mathbb{R}^m$ is the Jacobian $m\times m$ matrix
$$ \rmD \varphi(t,\omega,x) = \frac{\partial \varphi(t, \omega,x)}{\partial x}\,.$$
Differentiating the equation
$$ \varphi(t+s,\omega,x) = \varphi(t, \theta_s \omega, \varphi(s,\omega,x))$$
on both sides and applying the chain rule to the right hand side yields
$$ \rmD \varphi(t+s,\omega,x) = \rmD \varphi(t, \theta_s \omega, \varphi(s,\omega,x))\rmD \varphi(s,\omega,x) =  \rmD \varphi(t, \Theta_s(\omega,x))\rmD \varphi(s,\omega,x)\,,$$
i.e.~the cocycle property of the fiberwise mappings with respect to the skew product maps $(\Theta_t)_{t \in \mathbb T}$ (see equation~\eqref{eq:skewproductflow}). Let us further assume that the random dynamical system possesses an \emph{invariant measure} $\mu$ (see Appendix~\ref{sec:invariant_measures}). This implies that $(\Theta,\rmD \varphi)$ is a random dynamical system with linear cocycle $\rmD \varphi$ over the metric dynamical system $(\Omega \times \mathbb{R}^m, \mathcal F \times \mathcal{B}(\mathbb{R}^m), (\Theta_t)_{t \in \mathbb T})$, see e.g. \cite[Proposition 4.2.1]{a98}. 

The main models in this article are stochastic differential equation in Stratonovich form
\begin{equation} \label{SDEgen}
\rmd X_t = b(X_t) \rmd t + \sum_{i=1}^n \sigma_i(X_t) \circ \rmd W_t^i,\qquad X_0=x \in \mathbb{R}^m,
\end{equation}
where $W_t^i$ are independent real valued Brownian motions, $b$ is a $C^k$ vector field, $k \geq 1$, and $\sigma_1, \dots, \sigma_n$ are $C^{k+1}$ vector fields satisfying bounded growth conditions, as e.g. (global) Lipschitz continuity, in all derivatives to guarantee the existence of a (global) random dynamical system for $\varphi$ and $\rmD \varphi$. We write the equation in Stratonovich form when differentiation is concerned as the classical rules of calculus are preserved. We can apply the conversion formula to the It\^{o} integral to obtain the situation of $\eqref{SDE_RDS}$.
According to \cite{as95}, the derivative $\rmD \varphi(t,\omega,x)$ applied to an initial condition $v_0 \in \mathbb{R}^m$ solves uniquely the variational equation given by
\begin{equation}\label{varproblem}
  \rmd v = \rmD b(\varphi(t,\omega)x) v \,\rmd t + \sum_{i=1}^n \rmD \sigma_i(\varphi(t,\omega)x) v \circ \rmd W_t^i \,, \quad \text{where } v \in \mathbb{R}^m\,.
\end{equation}

The hyperbolicity of such a differentiable RDS with ergodic invariant measure $\mu$ and random cycle $A$ is expressed via its Lyapunov spectrum which is given due to the Multiplicative Ergodic Theorem (MET) (see Theorem~\ref{MET} in Appendix~\ref{LyapSpec}) under the integrability assumption
\begin{equation} \label{integrabilityassumpt}
\sup_{0 \leq t \leq 1} \log^+ \| \rmD \varphi(t, \omega, \cdot) \| \in L^1(\mu),
\end{equation}
where $\| \rmD\varphi(t, \omega, \cdot) \| $ denotes the operator norm of the Jacobian as a linear operator from $\txtT_x \mathbb{R}^m$ to $\txtT_{\varphi(t, \omega,x)} \mathbb{R}^m $ induced by the Euclidean norm and $ \log^+(a) = \max \{\log(a);0\}$.

Analogously to the characteristic exponents discussed for the deterministic case in Section~\ref{sec:deterministiccase}, the spectrum of $p \leq m$ \emph{Lyapunov exponents} $\lambda_1 > \lambda_2 > \dots > \lambda_p$ quantifies the asymptotic exponential rates of infinitesimally close trajectories.
\subsection{Random attractors} \label{sec:randomattractors}
Let $(\theta, \varphi)$ be a white noise random dynamical system on $\mathbb{R}^m$. (Note that the following can be formulated more generally in complete metric spaces $(\cX, d)$ but that we again restrict ourselves to the Euclidean case for reasons of clarity). Due to the non-autonomous nature of the RDS, there are no fixed attractors for dissipative systems and different notions of a random attractor exist.
We introduce these related but different definitions of random attractors in the following, with respect to tempered sets. 
Specific random attractors, \emph{attracting random cycles}, will play a crucial role in the following chapters. A random variable $R:\Omega\rightarrow \mathbb{R}$ is called \emph{tempered} if
\[
\lim_{t \to \pm \infty} \frac{1}{|t|} \ln^{+} R(\theta_t\omega)=0 \faa \omega\in \Omega\,,
\]
see also \cite[p.~164]{a98}. A set $D\in \mathcal F\otimes \mathcal B(\R^m)$ is called \emph{tempered} if there exists a tempered random variable $R$ such that
\[
  D(\omega)\subset B_{R(\omega)}(0) \faa \omega\in\Omega\,,
\]
where $B_{R(\omega)}(0)$ denotes a ball centered at zero with radius $R(\omega)$ and $D(\omega):=\{x\in  \R^m: (\omega, x)\in D\}$. $D$ is called compact if $D(\omega)\subset \R^m$ is compact for almost all $\omega\in\Omega$.
Denote by $\mathcal D$ the set of all compact tempered sets $D\in \mathcal F\otimes \mathcal B(\R^m)$ and by
$$\operatorname{dist}(E, F):= \sup_{x\in E}\inf_{y\in F} d(x,y)$$
the \emph{Hausdorff seperation} or \emph{semi-distance}, where $d$ denotes again the Euclidean metric. 
We now define different notions of a random attractor with respect to a family of sets $\mathcal S \subset \mathcal D$, see also \cite[Definition~14.3]{rk04} and \cite[Definition 15]{crauelkloeden15}.
\begin{defn}[Random attractor] \label{randomattractor}
   The set $A \in\mathcal S \subset \mathcal D$ that is strictly $\varphi$-invariant, i.e.
   \begin{equation*}
		\varphi(t,\omega) A(\omega) = A (\theta_t  \omega) \fa t\ge 0\mbox{ and almost all } \omega \in \Omega\,,
		\end{equation*}
is called  
	\begin{enumerate}
		\item[(i)] a \emph{random pullback attractor} with respect to $\mathcal S$ if for all $D \in \mathcal S$ we have
		\begin{equation*}
		\lim_{t \to \infty} \operatorname{dist} \big(\varphi(t, \theta_{-t} \omega)D(\theta_{-t}\omega), A(\omega)\big) = 0 \faa \omega\in\Omega\,,
		\end{equation*}
        \item[(ii)] a \emph{random forward attractor} with respect to $\mathcal S$ if for all $D \in \mathcal S$ we have
		\begin{equation*}
		\lim_{t \to \infty} \operatorname{dist} \big(\varphi(t, \omega)D(\omega), A(\theta_t\omega)\big) = 0 \faa \omega\in\Omega\,,
		\end{equation*}
		\item[(iii)] a \emph{weak random attractor} if it satisfies the convergence property in (i) (or (ii)) with almost sure convergence replaced by convergence in probability,
		\item[(iv)] a \emph{(weak) random (pullback or forward) point attractor} if it satisfies the corresponding properties above for $\mathcal{S}= \{ D \subset \R^m \,: \, D = \{y\}  \text{ for some } y \in \R^m\}$, i.e.~for single points $y \in \R^m$.
	\end{enumerate}
\end{defn}
Note that due to the $\mathbb{P}$-invariance of $\theta_t$ for all $t \in \mathbb{R}$, it is easy to derive that weak attraction in the pullback and the forward sense are the same and, hence, the notion of a weak random attractor in Definition~\ref{randomattractor}~(iii) is consistent. However, random pullback attractors and random forward attractors with almost sure convergence, as defined above, are generally not equivalent (see \cite{Scheutzow02} for counter-examples). In the following, we will be careful with this distinction, yet in our main examples the random pullback attractor and random forward attractor will be the same. In this case we will simply speak of \emph{the random attractor}.

Before we introduce random cycles and random periodic solutions, we add some remarks on Definition~\ref{randomattractor}.
\begin{remark}
Note that we require that the random attractor is measurable with respect to $\mathcal F\otimes \mathcal B(\R^m)$, in contrast to a weaker statement often used in the literature (see also \cite[Remark~4]{crauelkloeden15}).
\end{remark}
\begin{remark}
In many cases, the family of sets $\mathcal{S}$ is chosen to be the family of all bounded or compact (deterministic) subsets $B \subset \R^m$, as for example in \cite{fgs16}. Note that our definition of random attractors is a generalization of this weaker definition.
\end{remark}

\subsubsection{Attracting random cycles and random periodic solutions}
Consider a random dynamical system $(\theta, \varphi)$ on $\R^m$. In the situation of a deterministic limit cycle, the limit cycle is the attractor for all subsets of a neighbourhood of this attractor. Analagously, we give the following definition for the random setting.
\begin{defn}[Attracting Random Cycle] \label{defn:random_limit_cycle}
We call a random (forward or pullback) attractor $A$ for $(\theta, \varphi)$, with respect to a collection of sets $\mathcal{S}$, an \emph{attracting random cycle} if for almost all $\omega \in \Omega$ we have $A(\omega) \cong S^1$, i.e. every fiber is homeomorphic to the circle.
\end{defn}
Furthermore, we need to find a stochastic analogue to the limit cycle as a periodic orbit. Firstly, we follow \cite{ZhaoZheng09} for introducing the notion of random periodic solutions: 

\begin{defn}[Random periodic solution] \label{RPS_def}
Let $\mathbb{T} \in \{ \R, \R_0^+ \}$.
A \emph{random periodic solution} is an $\mathcal{F}$-measurable periodic function $\psi : \Omega \times \mathbb{T} \to \R^m$ of period $T>0$ such that for all $\omega \in  \Omega$
\begin{equation} \label{RPS}
\psi( t + T, \omega) = \psi(t, \omega) \ \text{ and } \ \varphi(t,\omega, \psi(t_0,\omega )) = \psi(t + t_0,\theta_t \omega ) \ \text{ for all } t,t_0 \in \mathbb{T}\,.
\end{equation}
\end{defn}
Note that this definition assumes that $T\in\R$ does not depend on the noise realization $\omega$. We will see the limitations of that concept in Example~\ref{ex:nonfixedphase}, extending the following example which we introduce first.

\begin{ex} \label{ex:basic_fixed_per}
Similarly to \cite{ZhaoZheng09}, consider the planar stochastic differential equation
\be
\label{eq:exHopf}
\begin{array}{lcl}
\rmd x &=& \left( x - y - x\left(x^2 + y^2\right) \right)\, \rmd t + \sigma x \circ \rmd W_t\,,\\
\rmd y &=&  \left( x + y - y\left(x^2 + y^2\right) \right)\, \rmd t + \sigma y \circ \rmd W_t\,.
\end{array}
\ee
where $\sigma \geq 0$, $W_t$ denotes a one-dimensional standard Brownian motion and the noise is of Stratonovich type. 
We denote the cocycle of the induced random dynamical system by $\varphi = (\varphi_1, \varphi_2)$.
Equation~\eqref{eq:exHopf} can be transformed into polar coordinates $(\vartheta, r) \in [0, 2 \pi) \times [0, \infty)$
\be
\label{eq:exzz}
\begin{array}{lcl}
\rmd \vartheta &=& 1 \, \rmd t,\\
\rmd r &=&  (r - r^3) \, \rmd t + \sigma r \circ \, \rmd W_t\,.
\end{array}
\ee
Therefore, in the situation without noise ($\sigma = 0$), the system is as in Example~\ref{ex:Guck} with $h\equiv 1$ and attracting limit cycle at radius $r=1$.
With noise switched on ($\sigma > 0$),
equation~\eqref{eq:exzz} has an explicit unique solution given by
$$ \hat \varphi(t, \omega, (\vartheta_0, r_0)) = \left(\vartheta_0 + t \mod 2 \pi, \ \frac{r_0 \txte^{t + \sigma W_t(\omega)}}{\left( 1 + 2 r_0^2 \int_0^t e^{2(s + \sigma W_s(\omega))} \rmd s \right)^{1/2} } \right) =: ( \vartheta(t, \omega, \vartheta_0), r(t,\omega, r_0)) \,.$$
Moreover, there is a stationary solution for the radial component, satisfying $r(t, \omega, r^*(\omega)) = r^*(\theta_t \omega)$, and given by
\begin{equation} \label{eq:rstar}
r^*(\omega) = \left( 2 \int_{-\infty}^{0} \txte^{2s + 2 \sigma W_s(\omega)} \rmd s \right)^{-1/2}\,.
\end{equation}
Furthermore, one can see from a straightforward computation that for all $(x,y) \neq (0,0)$ and almost all $\omega \in \Omega$
$$ \left(\varphi_1(t, \theta_{-t} \omega,x)^2 + \varphi_2(t, \theta_{-t} \omega, y)^2\right)^{1/2} \to r^* (\omega)\ \text{ as } t \to \infty \,,$$
and also
$$ \left(\varphi_1(t, \omega,x)^2 + \varphi_2(t,\omega, y)^2\right)^{1/2} \to r^* (\theta_t \omega)\ \text{ as } t \to \infty \,.$$
Hence, the planar system~\eqref{eq:exHopf} has a random attractor $A$ in the pullback and forward sense, with respect to $\mathcal{S} =\mathcal{D} \setminus \{ \{0\} \}$, where $\mathcal D$ denotes the set of all compact tempered sets $D\in \mathcal F\otimes \mathcal B(\R^2)$ (see also Section~\ref{appendix:randomattractor}), and the fibers of $A$ are given by (see Figure~\ref{fig:attractors})
\be \label{randomattractor_ex}
 A(\omega) = \{ r^* (\omega)(\cos \alpha, \sin \alpha) \,: \, \alpha \in [0, 2 \pi) \}.
 \ee
The system possesses, for any fixed $\vartheta_0 \in [0, 2 \pi)$, the random periodic solution $\psi$ which is defined by
$$ \psi(t, \omega) = r^* (\omega)(\cos ( \vartheta_0 + t), \sin (\vartheta_0 +t))\,.$$
Indeed, it is easy to check that  $\psi(t, \omega)  = \psi(t + 2 \pi, \omega) $ and
$ \varphi(t, \omega, \psi(t_0, \omega)) = \psi(t + t_0, \theta_t \omega)$ for all $t, t_0 \geq 0$.
\end{ex}

\begin{ex}\label{ex:nonfixedphase}
(a) Now consider a stochastic version of Example~\ref{ex:Guck} when the phase dynamics depends on the amplitude, i.e.
\be
\label{eq:exzz2}
\begin{array}{lcl}
\rmd \vartheta &=& h(r) \, \rmd t,\\
\rmd r &=&  (r - r^3) \, \rmd t + \sigma r \circ \, \rmd W_t\,,
\end{array}
\ee
where the smooth function $h: \R \to \R$ with $h \geq K_h > 0$ is non-constant.
The random attractor $A$ for the corresponding planar system 
\be
\label{eq:exzz2planar}
\begin{array}{lcl}
\rmd x &=& \left( x - h\left( \sqrt{x^2 + y^2}\right)y - x\left(x^2 + y^2\right) \right)\, \rmd t + \sigma x \circ \rmd W_t\,,\\
\rmd y &=&  \left( h\left( \sqrt{x^2 + y^2}\right) x + y - y\left(x^2 + y^2\right) \right)\, \rmd t + \sigma y \circ \rmd W_t\,.
\end{array}
\ee
is exactly the same as before, as illustrated in Figure~\ref{fig:attractors}. 
We observe for a point $ a(\omega) : = r^*(\omega)(\cos \vartheta_0, \sin \vartheta_0) \in A(\omega)$, where $r^*$ is the random variable defined in equation~\eqref{eq:rstar} and $ \vartheta_0 \in [0, 2 \pi)$, that the cocycle satisfies
$$ \varphi(t, \omega, a(\omega)) = r^*(\theta_t \omega)\left(\cos \left( \vartheta_0 + \int_0^t h(r^*(\theta_s \omega)) \rmd s \right),\, \sin \left( \vartheta_0 + \int_0^t h(r^*(\theta_s \omega)) \rmd s\right) \right)\,.$$
There cannot be a random periodic solution in the sense of Definition~\ref{RPS_def}, since noise-independent periodicity is not possible if $h$ is non-constant. 

\medskip
\noindent (b) Naturally, we can also consider the case where the phase amplitude is additionally perturbed by noise, i.e.
\be
\label{eq:exzz2_noise}
\begin{array}{lcl}
\rmd \vartheta &=& h(r) \, \rmd t + \tilde h(r) \circ \, \rmd W_t^2,\\
\rmd r &=&  (r - r^3) \, \rmd t + \sigma r \circ \, \rmd W_t^1\,,
\end{array}
\ee
where $W_t = (W_t^1, W_t^2)$ is now two-dimensional Brownian motion and $h, \tilde h: \R \to \R$ are smooth functions. 
\end{ex}
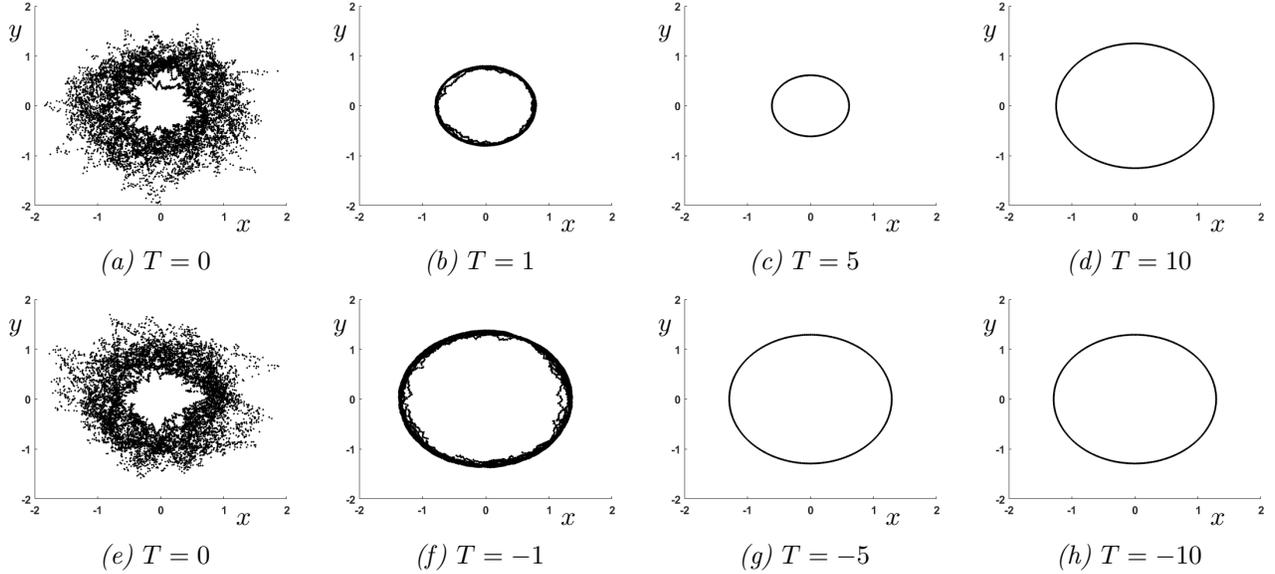
\begin{figure}[htpb]
\centering
\begin{subfigure}[b]{.25\textwidth}
  \centering
  \begin{overpic}[width=1\linewidth]{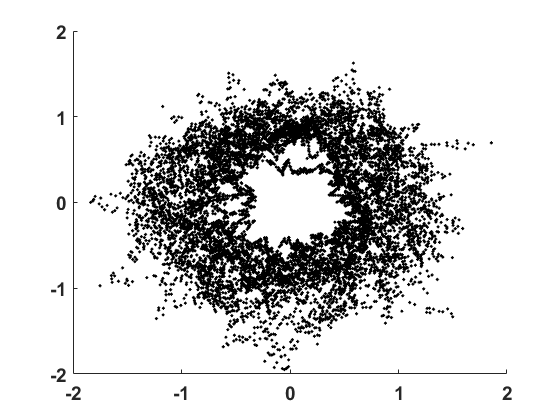}           
                \put(75,0){\footnotesize $x$}
                \put(5,60){\footnotesize $y$}
        \end{overpic}
  \caption{$T=0$}
\end{subfigure}%
\begin{subfigure}[b]{.25\textwidth}
  \centering
\begin{overpic}[width=1\linewidth]{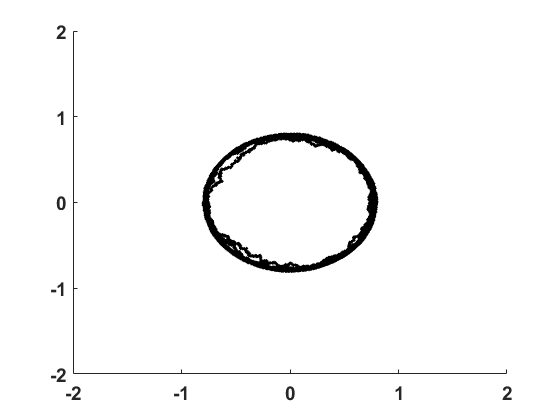}           
                \put(75,0){\footnotesize $x$}
                \put(5,60){\footnotesize $y$}
        \end{overpic}
  \caption{$T=1$}
\end{subfigure}%
\begin{subfigure}[b]{.25\textwidth}
  \centering
\begin{overpic}[width=1\linewidth]{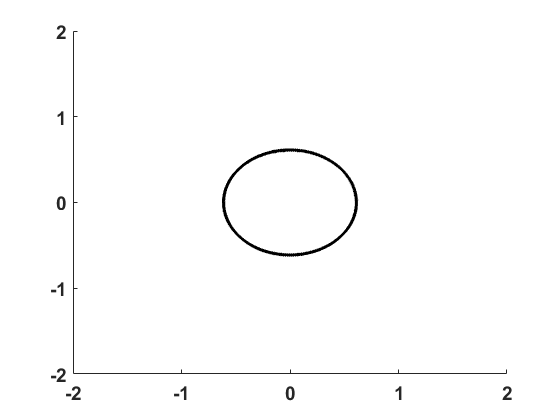}           
                \put(75,0){\footnotesize $x$}
                \put(5,60){\footnotesize $y$}
        \end{overpic}
  \caption{$T=5$}
\end{subfigure}%
\begin{subfigure}[b]{.25\textwidth}
  \centering
  \begin{overpic}[width=1\linewidth]{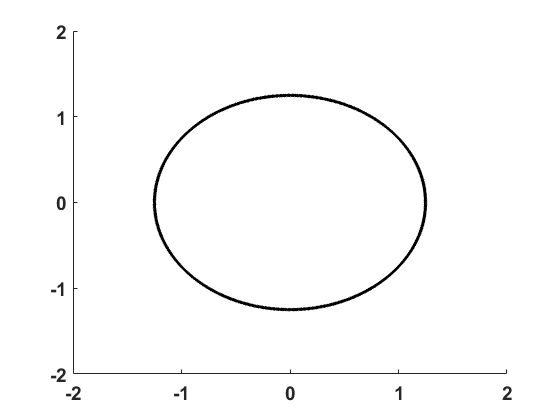}           
                \put(75,0){\footnotesize $x$}
                \put(5,60){\footnotesize $y$}
        \end{overpic}
  \caption{$T=10$}
\end{subfigure}%
\hfill
\begin{subfigure}{.25\textwidth}
  \centering
  \begin{overpic}[width=1\linewidth]{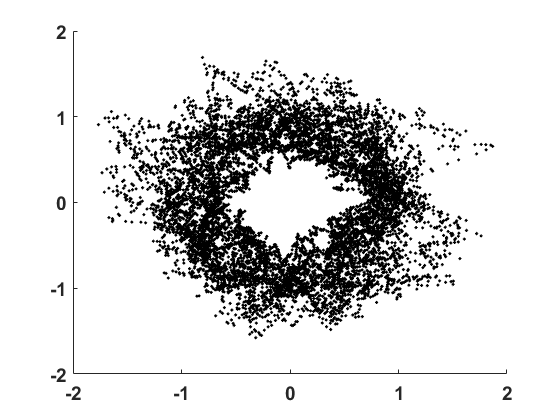}          
                \put(75,0){\footnotesize $x$}
                \put(5,60){\footnotesize $y$}
        \end{overpic}
  \caption{$T=0$}
\end{subfigure}%
\begin{subfigure}{.25\textwidth}
  \centering
    \begin{overpic}[width=1\linewidth]{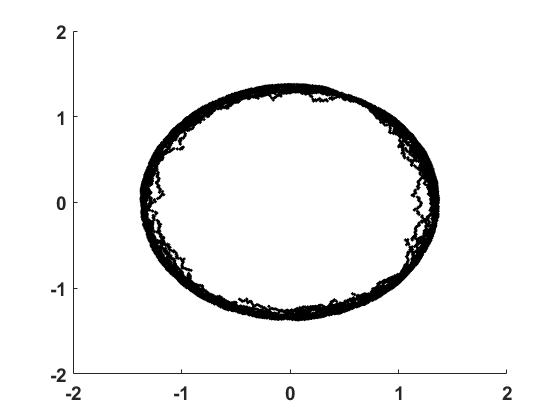}          
                \put(75,0){\footnotesize $x$}
                \put(5,60){\footnotesize $y$}
        \end{overpic}
  \caption{$T=-1$}
\end{subfigure}%
\hfill
\begin{subfigure}{.25\textwidth}
  \centering
    \begin{overpic}[width=1\linewidth]{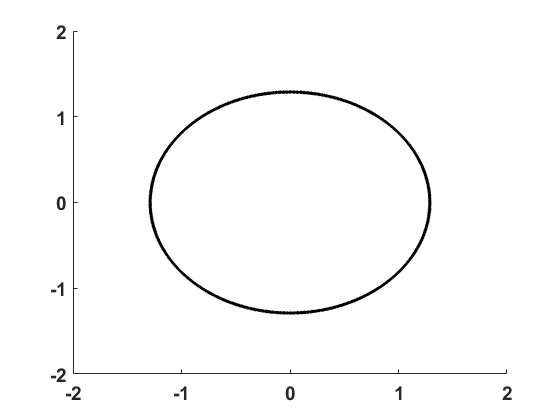}          
                \put(75,0){\footnotesize $x$}
                \put(5,60){\footnotesize $y$}
        \end{overpic}
  \caption{$T=-5$}
\end{subfigure}%
\begin{subfigure}{.25\textwidth}
  \centering
    \begin{overpic}[width=1\linewidth]{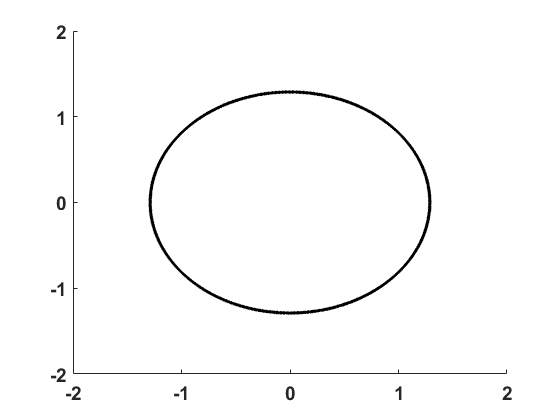}          
                \put(75,0){\footnotesize $x$}
                \put(5,60){\footnotesize $y$}
        \end{overpic}
  \caption{$T=-10$}
\end{subfigure}%
\caption{
Numerical simulations in $(x,y)$-coordinates, using Euler-Marayama integration with step size $\rmd t =10^{-2}$, of forward and pullback dynamics of system \eqref{eq:exHopf} for a set $B$ of initial conditions generated by a trajectory of~\eqref{eq:exHopf} ((a) and (e)). In (b)--(d), we show the numerical approximation of $\varphi(T,\omega, B)$ for some $\omega \in \Omega$, approaching the fiber $A(\theta_T \omega)$ of the random attractor, changing in forward time. In (f)--(h), we show the numerical approximation of $\varphi(-T,\theta_{-T} \omega, B)$ for some $\omega \in \Omega$, approaching the fiber $A(\omega)$ of the random attractor, fixed by the pullback mechanism.}
\label{fig:attractors}
\end{figure}

Example~\ref{ex:nonfixedphase} motivates us to introduce the following notion of a more general form of random periodic solution. The potential relevance of finding such a generalization was first discussed by Hans Crauel\footnote{Through personal communication.}; hence, we have chosen the name.

\begin{defn}[Crauel random periodic solution] \label{def:CRPS}
Let $\mathbb{T} \in \{ \R, \R_0^+ \}$.
A \emph{Crauel random periodic solution} (CRPS) is a pair $(\psi,T)$ consisting of $\mathcal{F}$-measurable functions $\psi : \Omega \times \mathbb{T} \to \R^m$ and $T : \Omega \to \R$ such that for all $\omega \in  \Omega$
\begin{equation} \label{eq:RPS_gen}
\psi( t, \omega) = \psi(t + T(\theta_{-t} \omega), \omega) \ \text{ and } \ \varphi(t,\omega, \psi(t_0,\omega )) = \psi(t + t_0,\theta_t \omega ) \ \text{ for all } t,t_0 \in \mathbb{T}\,.
\end{equation}
\end{defn}
In particular, note that condition~\eqref{eq:RPS_gen} implies $\psi(0, \omega) = \psi(T(\omega), \omega)$ (see Figure~\ref{fig:rps_sketch} for further details). Furthermore, observe that the classical random periodic solution according to Definition \ref{RPS_def} is simply a Crauel random periodic solution with constant $T$. We show that Definition~\ref{def:CRPS} applies to system~\eqref{eq:exzz2}, demonstrating the suitability of this definition.

\begin{prop} \label{prop:crps_ex}
(a) The planar system associated with~\eqref{eq:exzz2} has a family of Crauel random periodic solutions $(\psi_{\vartheta},T)$ which is defined for every $\vartheta \in [0, 2 \pi)$ by
\begin{equation} \label{eq:CRPS_ex}
 \psi_{\vartheta}(t, \omega) = r^* (\omega)\left(\cos \left( \vartheta + \int_{-t}^0 h(r^*(\theta_s \omega)) \rmd s \right), \ \sin \left( \vartheta + \int_{-t}^0 h(r^*(\theta_s \omega)) \rmd s \right)\right)\,,
\end{equation}
and 
\begin{equation} \label{eq:int_Tomega}
\int_{-T( \omega)}^{0} h(r^*(\theta_s \omega)) \rmd s = 2 \pi\,,
\end{equation} 
for almost all $\omega \in \Omega$ and all $ t \in \R_0^+$.

(b) The system associated with~\eqref{eq:exzz2_noise} has a family of Crauel random periodic solutions $(\psi_{\vartheta},T)$ which is defined for every $\vartheta \in [0, 2 \pi)$ by $\psi_{\vartheta}$ analogously to \eqref{eq:CRPS_ex}, just adding $\int_{-t}^{0} \tilde h(r^*(\theta_s \omega)) \circ \,  \rmd W_s^2(\omega)$ to the angular direction,
and 
\begin{equation} \label{eq:int_Tomega_noise}
T(\omega) = \inf \left\{ t > 0: \left|\int_{-t}^{0} h(r^*(\theta_s \omega)) \rmd s +  \int_{-t}^{0} \tilde h(r^*(\theta_s \omega)) \circ \,  \rmd W_s^2(\omega)\right| = 2 \pi \right\}\,.
\end{equation} 
for almost all $\omega \in \Omega$ and all $ t \in \R_0^+$.
\end{prop}
\begin{proof}
Without loss of generality let $\vartheta=0$.

(a) The fact that $T: \Omega \to \mathbb{R}$ is well defined can be seen as follows: fix $\omega \in \Omega$ and let 
$$g_{\omega}(t) = \int_{-t}^{0} h(r^*(\theta_s \omega)) \rmd s - 2 \pi.$$
Then $g_{\omega}(0) < 0$ and $g_{\omega}(2 \pi/K_h) > 0$ and, hence, the existence of $T(\omega)$ follows from the intermediate value theorem. Moreover, we have by a change of variables that
$$  2 \pi = \int_{-T( \theta_{-t} \omega)}^{0} h(r^*(\theta_{s-t} \omega)) \rmd s =\int_{-(t + T(\theta_{-t} \omega))}^{-t} h(r^*(\theta_s \omega)) \rmd s \,.$$
We use this observation to conclude that for almost all $\omega \in \Omega$ and any $t \geq 0$
\begin{align*}
\psi(t + T(\theta_{-t} \omega), \omega) &= r^* (\omega)\left(\cos \left( \int_{-(t + T(\theta_{-t} \omega))}^{0} h(r^*(\theta_s \omega)) \rmd s \right), \ \sin \left(  \int_{-(t + T(\theta_{-t} \omega))}^{0} h(r^*(\theta_s \omega)) \rmd s\right)\right) \\
&= r^* (\omega)\left(\cos \left( 2 \pi + \int_{-t }^{0} h(r^*(\theta_s \omega)) \rmd s \right), \ \sin \left(2 \pi +   \int_{-t}^{0} h(r^*(\theta_s \omega)) \rmd s\right)\right) \\
&=  \psi(t, \omega)\,.
\end{align*}
Furthermore, we observe that for almost all $\omega \in \Omega$ and $t, t_0 \geq 0$
\begin{align*}
\varphi(t, \omega,\psi(t_0, \omega)) &= r^* (\theta_t \omega)\left(\cos \left( \int_{-t_0}^{t} h(r^*(\theta_s \omega)) \rmd s \right), \ \sin \left(  \int_{-t_0}^{t} h(r^*(\theta_s \omega)) \rmd s\right)\right) \\
&= r^* (\theta_t \omega)\left(\cos \left( \int_{-t_0 -t}^{0} h(r^*(\theta_{s+t} \omega)) \rmd s \right), \ \sin \left(  \int_{-t_0 - t}^{0} h(r^*(\theta_{s+t} \omega)) \rmd s\right)\right)  \\
&=  \psi(t + t_0, \theta_t \omega)\,.
\end{align*}
(b) The fact that $T: \Omega \to \mathbb{R}$ is well defined almost surely in this case follows directly from the properties of SDEs on compact intervals, in this case $[-2 \pi, 2 \pi]$.
Moreover, we have by a change of variables that
\begin{align*}
 2 \pi &= \left|\int_{-T( \theta_{-t} \omega)}^{0} h(r^*(\theta_{s-t} \omega)) \rmd s +  \int_{-T( \theta_{-t} \omega)}^{0} \tilde h(r^*(\theta_{s-t} \omega)) \circ \,  \rmd W_s^2(\theta_{-t} \omega)\right| \\
 &= \left|\int_{-T( \theta_{-t} \omega)}^{0} h(r^*(\theta_{s-t} \omega)) \rmd s +  \int_{-T( \theta_{-t} \omega)}^{0} \tilde h(r^*(\theta_{s-t} \omega)) \circ \,  \rmd W_{s-t}^2( \omega)\right| \\
 &= \left| \int_{-(t + T(\theta_{-t} \omega))}^{-t} h(r^*(\theta_s \omega)) \rmd s +  \int_{-(t+T( \theta_{-t} \omega))}^{-t} \tilde h(r^*(\theta_{s} \omega)) \circ \,  \rmd W_{s}^2( \omega)\right| \,.
\end{align*} 
We use this observation to conclude 
$\psi(t + T(\theta_{-t} \omega), \omega) = \psi(t, \omega)$ as in (a).
Furthermore, we observe that for almost all $\omega \in \Omega$ and $t, t_0 \geq 0$
\begin{align*}
 \int_{-t_0}^{t} \tilde h(r^*(\theta_{s} \omega)) \circ \,  \rmd W_{s}^2( \omega) &= \int_{-t_0 -t}^{0} \tilde h(r^*(\theta_{s+t} \omega)) \circ \,  \rmd W_{s+t}^2( \omega) \rmd s  \\
  &= \int_{-t_0 -t}^{0} \tilde h(r^*(\theta_{s} (\theta_t \omega))) \circ \,  \rmd W_{s}^2( \theta_t \omega) \rmd s \,,
\end{align*}
such that $\varphi(t, \omega,\psi(t_0, \omega)) = \psi(t + t_0, \theta_t \omega)$ follows as in (a).
This finishes the proof.
\end{proof}
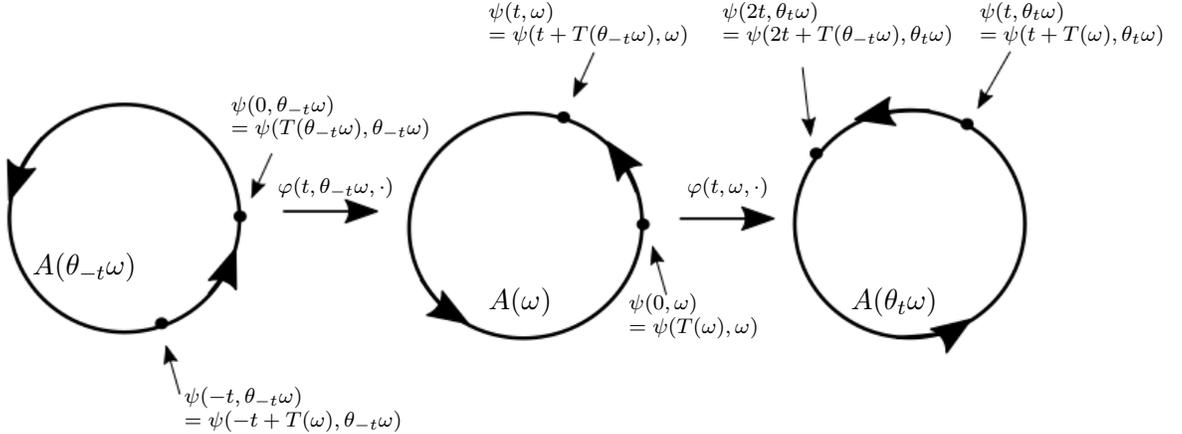
\begin{figure}[htbp]
        \centering
  		\begin{overpic}[width=0.9\textwidth]{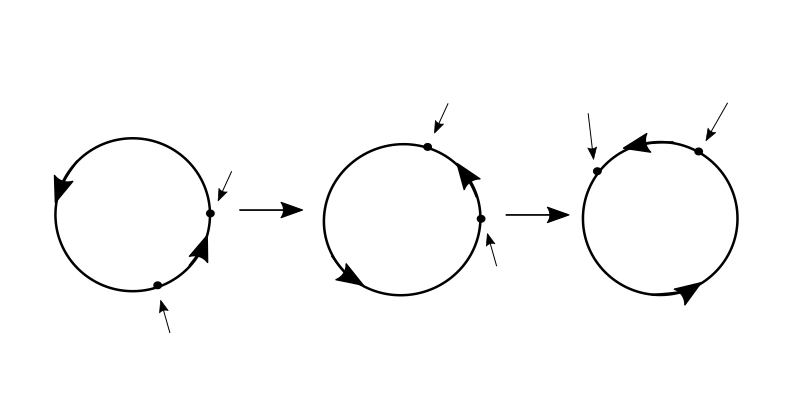} 
  		         \put(9,20){\small $A(\theta_{-t} \omega)$ }
  		         \put(48,17){\small $A(\omega)$ }
  		         \put(79,17){\small $A(\theta_{t} \omega)$ }
                \put(26,34){\scriptsize $\psi(0, \theta_{-t} \omega)$ }
                \put(26,32){\scriptsize $= \psi(T(\theta_{-t} \omega), \theta_{-t} \omega) $}
                \put(60,17){\scriptsize $\psi(0, \omega)$ }
                \put(60,15){\scriptsize $= \psi(T(\omega), \omega) $}
                \put(22,9){\scriptsize $\psi(-t, \theta_{-t} \omega)$ }
                \put(22,7){\scriptsize $= \psi(-t+T(\omega), \theta_{-t} \omega) $}
                \put(30,27){\scriptsize $\varphi(t,\theta_{-t }\omega, \cdot) $}
                \put(65,27){\scriptsize $\varphi(t,\omega, \cdot) $}
                \put(48,42){\scriptsize $\psi(t, \omega)$ }
                \put(48,40){\scriptsize $= \psi(t + T(\theta_{-t }\omega), \omega) $}
                \put(68,42){\scriptsize $\psi(2t, \theta_{t} \omega)$ }
                \put(68,40){\scriptsize $= \psi(2t + T(\theta_{-t }\omega), \theta_{t}  \omega) $}
                \put(90,42){\scriptsize $\psi(t, \theta_{t} \omega)$ }
                \put(90,40){\scriptsize $= \psi(t + T(\omega), \theta_{t}  \omega) $}
        \end{overpic}
\caption{Sketch of Crauel random periodic solutions (CRPS), following two points along the dynamics from $A(\theta_{-t} \omega)$ via $A(\omega)$ to $A(\theta_t \omega)$. The point $\psi(0, \theta_{-t} \omega)$ is mapped by $\varphi(t, \theta_{-t} \omega, \cdot)$ to $\psi(t, \omega)$ which is then mapped by $\varphi(t, \omega, \cdot)$ to $\psi(2t, \theta_t\omega)$, in each case preserving the period $T(\theta_{-t} \omega)$. Similarly, the point $\psi(-t, \theta_{-t} \omega)$ is mapped by $\varphi(t, \theta_{-t} \omega, \cdot)$ to $\psi(0, \omega)$ which is then mapped by $\varphi(t, \omega, \cdot)$ to $\psi(t, \theta_t\omega)$, in each case preserving the period $T(\omega)$. The arrows indicate that the CRPS parametrizes the fiber of the attractor as $A(\omega) = \{ \psi(t, \omega)\,: \, t \in [0,  T(\omega))\}$. }
        \label{fig:rps_sketch}
\end{figure}
Note that in Example~\ref{ex:nonfixedphase}, and by that also the simpler subcase Example~\ref{ex:basic_fixed_per}, it is easy to check that the Lyapunov exponents satisfy $\lambda_1 = 0$ and $\lambda_2 < 0$.
We want to make three additional remarks on Proposition~\ref{prop:crps_ex}, also concerning Definition~\ref{def:CRPS}.
\begin{remark} \label{rem:crps_alternatives}
The proof of Proposition~\ref{prop:crps_ex} shows why we require $\psi(t+T(\theta_{-t}\omega), \omega) = \psi(t, \omega)$ in Definition~\ref{def:CRPS} instead of choosing $T(\omega)$ or $T(\theta_t \omega)$ in such a formula. It is precisely the relation we obtain from equations~\eqref{eq:CRPS_ex} and~\eqref{eq:int_Tomega}. 
Instead of equation~\eqref{eq:int_Tomega}, one might alternatively consider 
\begin{equation*}
\int_{0}^{T( \omega)} h(r^*(\theta_s \omega)) \rmd s = 2 \pi\,,
\end{equation*} 
and replace the time integral in $\psi_{\vartheta}(t, \omega)$~\eqref{eq:CRPS_ex} accordingly.
However, it is easy to check that the invariance requirement $\varphi(t, \omega, \psi_{\vartheta}(t_0, \omega)) = \psi_{\vartheta}(t+t_0, \theta_t \omega)$ is not satisfied in this situation. Hence, the choice of period in Definition~\ref{def:CRPS} turns out to be the appropriate one for an application to Example~\ref{ex:nonfixedphase} which we see as the fundamental model for extending random periodic solutions to noise-dependent periods. Additionally note that, when $\tilde h \neq 0$ in equation~\eqref{eq:exzz2_noise}, the direction of periodicity depends on the noise realization $\omega$.
\end{remark}
\begin{remark} \label{rem:rps_attract}
Note that for any $\vartheta \in [0, 2 \pi)$ we have $\psi_{\vartheta} (t, \omega) \in A(\omega)$ for all $t \geq 0$, $\omega \in \Omega$, where $A$ is the random attractor given in equation~\eqref{randomattractor_ex}. Hence, we have established the analogous situation to the deterministic case in the sense that the attracting random cycle corresponds to a random periodic solution; see also Figure~\ref{fig:rps_sketch}.
\end{remark}
\begin{remark} \label{rem:nontractable_model}
One may ask what happens when $h, \tilde h$ in equation~\eqref{eq:exzz2_noise} also depend on $\vartheta$. Then there can, of course, still be a CRPS but we do not know a priori the existence of some stationary process $\vartheta^*$ similarly to $r^*$ which we need to write down for an explicit solution such as \eqref{eq:CRPS_ex}.
\end{remark}
We will see later in Proposition~\ref{prop:expect_period} that we can determine $\mathbb{E}[T(\omega)] < \infty$, using a variant of the \emph{Andronov-Vitt-Pontryagin formula} (cf.~\cite{schuss10}).
\subsubsection{Chaotic random attractors and singletons}
\label{sec:chaotic_attractors}
More generally, i.e., in addition to the case with first Lyapunov exponent $\lambda_1=0$, we want to consider the situations where $\lambda_1 > 0$ and $\lambda_1 < 0$ (always assuming volume contraction to an attractor expressed by $\sum_{j} \lambda_j < 0$). For $\lambda_1<0$, this typically means that the random attractor is a singleton (see, for example, \cite{fgs16}) and one speaks of complete \emph{synchronization}. In such a situation, the dynamics on the random attractor is trivial, so there is no natural notion of isochronicity. In the case $\lambda_1>0$, one typically speaks of a \emph{chaotic random attractor} which is not a singleton. We can illustrate these two cases by the following example very similar to the previous ones.

\begin{ex} \label{ex:chaos}
We consider the following stochastic differential equations on $\mathbb{R}^2$ with purely external noise of intensity $\sigma \geq 0$,
\begin{equation}\label{NormalForm}
\begin{array}{ll}
\rmd x &= (x -  y - (x-by)(x^2 + y^2))\rmd t + \sigma \circ \rmd W_t^1,\\
\rmd y &=  ( y + x - (bx+y)(x^2 + y^2)) \rmd t +  \sigma \circ \rmd W_t^2,
\end{array}
\end{equation}
where $ b \in \mathbb{R}$ and $W_t^1, W_t^2$
denote independent one-dimensional Brownian motions.
In polar coordinates the system can be written as
\begin{align} \label{Polar_original}
\rmd r &= \left( r  -  r^3  \right) \rmd t  + \sigma (\cos \vartheta \circ \, \rmd W_t^1 +  \sin \vartheta \circ \, \rmd W_t^2), \nonumber\\
\rmd \vartheta &= (1 + br^2) \, \rmd t + \frac{\sigma}{r}( -  \sin \vartheta \circ \, \rmd  W_t^1 +  \cos \vartheta \circ \, \rmd W_t^2).
\end{align}
This form illustrates the role of the parameter $b$ inducing a \emph{shear force}: if $b > 0$, the phase velocity $\frac{\rmd \vartheta}{\rmd t}$ depends on the amplitude $r$. 
Since Gaussian random vectors are invariant under orthogonal transformations, one might think of writing the problems with the independent Wiener processes
\begin{align*}
\rmd W_r &= \cos \vartheta \,\rmd W_t^1 + \sin \vartheta \, \rmd W_t^2,\\ 
\rmd W_\vartheta  &= - \sin \vartheta \, \rmd  W_t^1 + \cos \vartheta \, \rmd W_t^2.
\end{align*}
However, the pathwise properties of the processes seen as random dynamical systems change under this transformation. In~\eqref{Polar_original}, the radial components of the trajectories depend on $\vartheta$ which appears in the diffusion term and destroys the skew-product structure we had in the previous example~\ref{ex:nonfixedphase}.

It has been shown in \cite{delr17} that for $b$ small enough the first Laypunov exponent $\lambda_1 < 0$ is negative such that the corresponding random attractor $A$ is indeed a singleton. For $b$ large, one can see numerically that the attractor becomes chaotic. A proof of $\lambda_1 > 0$ has been obtained in \cite{EngelLambRasmussen1} for a simplified model of~\eqref{eq:polargeneral} in cylindrical coordinates and recently also in the setting of restricting the state space on a bounded domain and only considering the dynamics conditioned on survival in this domain, using a computer-assisted proof technique \cite{BredenEngel}.
\end{ex}

One can characterize chaotic random attractors as non-trivial geometric objects and supports of \emph{SRB measures}, i.e.~sample measures with densities on unstable manifolds. For details see \cite{BlumenthalYoung, LedrappierYoung} and for further discussions relevant for our setting e.g.~\cite{b12, Engelphdthesis}. Due to the compactness and the minimality property of random attractors there must be recurrence on these objects and one may even find Crauel Random Periodic Solutions there. However, it is questionable to what extent one can speak of isochronicity, given the very irregular recurrence properties. This already makes isochronicity a difficult issue for deterministic chaotic oscillators, see e.g.~\cite{Schwabedaletal2012}.

\subsection{Random limit cycles as normally hyperbolic random invariant manifolds} \label{sec:normally_hyperbolic}
As we have seen in Section~\ref{sec:chaotic_attractors}, we can generally not expect the persistence of periodic orbits from the deterministic to the stochastic case under (global) white noise perturbations. A point of view that is only considering local, bounded noise perurbations of normally hyperbolic manifolds, i.e.~implicitly also hyperbolic limit cycles, is presented in \cite{LiLuBates}, where normally hyperbolic random invariant manifolds and their foliations are studied.
In more details, consider the ODE~\eqref{eq:ODE} with a small random perturbation, i.e.~the \emph{random differential equation}
\be \label{eq:random_ODE}
\dot{x} = f(x) + \epsilon F(\theta_t \omega, x),
\ee
where $\epsilon > 0$ is a small parameter and $F$ is $C^1$, uniformly bounded in $x$, $C^0$ in $t$ for fixed $\omega$, and measurable in $\omega$. In several cases, SDEs can be transformed into a random differential equation~\eqref{eq:random_ODE}, in particular when the noise is additive or linear multiplicative; however, in this case, $F$ is generally not uniformly bounded. 
Hence, for an application of the following, one has to truncate the Brownian motion by a fixed large constant, as we will discuss later.
Let us firstly give the following definition:
\begin{defn}
A random invariant manifold for an RDS is a collection of nonempty closed random sets $\mathcal M (\omega)$, $\omega \in \Omega$, such that each $\mathcal{M}(\omega)$ is a manifold and
$$ \varphi(t, \omega, \mathcal M(\omega)) = \mathcal M(\theta_t \omega) \ \text{ for all } t \in \mathbb{R}, \omega \in \Omega.$$
The random invariant manifold $\mathcal M$ is called normally hyperbolic if for almost every $\omega \in \Omega$ and any $x \in \mathcal M(\omega)$, there exists a splitting which is $C^0$ in $x$ and measurable:
$$ \mathbb{R}^m = E^u(\omega, x) \oplus E^c(\omega, x) \oplus E^s(\omega, x)$$
of closed subspaces with associated projections $\Pi^u(\omega,x),\Pi^c(\omega,x)$ and $\Pi^s(\omega,x)$ such that
\begin{enumerate}[(i)]
\item the splitting is invariant
$$ \rmD \varphi(t,\omega,x) E^i(\omega,x) = E^i(\theta_t \omega, \varphi(t, \omega,x)), \ \text{ for } i=u,c,$$
and
$$ \rmD \varphi(t,\omega,x) E^s(\omega,x) \subset E^s(\theta_t \omega, \varphi(t, \omega,x)),$$
\item $\rmD \varphi(t, \omega,x)|_{E^i(\omega,x)} : E^i(\omega,x) \to E^i(\theta_t \omega, \varphi(t, \omega,x))$ is an isomorhpism for $i=u,c,s$ and $E^c(\omega,x)$ is the tangent space of $\mathcal M(\omega)$ at $x$,

\item there are $(\theta, \varphi)$-invariant random variables $\bar \alpha, \bar \beta: \mathcal M \to (0, \infty), \bar \alpha < \bar \beta$, and a tempered random variable $K(\omega,x) : \mathcal M \to [1, \infty)$ such that
\begin{align} 
\| \rmD \varphi(t,\omega,x) \Pi^s(\omega,x)\| &\leq K(\omega,x) e^{- \bar \beta(\omega,x) t} \ \text{ for } t\geq 0, \label{eq:exponents_stable}\\
\| \rmD \varphi(t,\omega,x) \Pi^u(\omega,x)\| &\leq K(\omega,x) e^{\bar \beta(\omega,x) t} \ \text{ for } t\leq 0, \\
\| \rmD \varphi(t,\omega,x) \Pi^c(\omega,x)\| &\leq K(\omega,x) e^{ \bar \alpha(\omega,x) \left|t\right|} \ \text{ for } - \infty < t < \infty. \label{eq:exponents_center}
\end{align}
\end{enumerate}
\end{defn}

We can then deduce the following statements:
\begin{prop} \label{prop:existence_of_random_cycles}
Assume that $\Phi$ is a $C^k$ flow, $k \geq 1$, in $\mathbb{R}^m$ which has a hyperbolic periodic orbit $\gamma$, with exponents $\bar \alpha =0 <  \bar \beta$ characterizing the normal hyperbolicity as in~\eqref{eq:exponents_stable}, \eqref{eq:exponents_center}. Then there exists a $\delta > 0$ such that for any random $C^1$ flow $\varphi(t, \omega, \cdot)$ in $\mathbb{R}^m$, as for example induced by an RDE~\eqref{eq:random_ODE}, with
$$ \|\Phi(t, \cdot) - \varphi(t, \omega, \cdot)\|_{C^1} < \delta, \ \text{ for all } t \in [0,1], \omega \in \Omega, $$
we have that
\begin{enumerate}[(i)]
\item the random flow $\varphi(t, \omega, \cdot)$ has a $C^1$ normally hyperbolic invaraint random manifold $\mathcal M(\omega)$ in a small neighbourhood of $\gamma$,
\item if $\varphi(t, \omega, \cdot)$ is $C^k$, then $\mathcal M(\omega)$ is a $C^k$ manifold diffeomorphic to $\gamma$ for each $\omega \in \Omega$,
\item there exists a stable manifold $\mathcal W^s(\omega)$ of $\mathcal M(\omega)$ under $\varphi(t, \omega, \cdot)$, i.e. for all $x \in \mathcal W^s(\omega)$
$$ \lim_{t \to \infty} \operatorname{dist} \big(\varphi(t, \omega, x), \mathcal M(\theta_t\omega)\big) = 0 \faa \omega\in\Omega$$
\item the manifold $\mathcal M(\omega)$ is, in fact, a random limit cycle in the sense of Definition~\ref{defn:random_limit_cycle}.
\end{enumerate}
\end{prop}
\begin{proof}
The statements (i)-(iii) follow directly from \cite[Theorem 2.2]{LiLuBates}. It is clear from (iii) that $\mathcal M(\omega)$ is a random forward attractor with respect to the collection $\mathcal S$ of tempered random sets whose fibers $S(\omega)$ are contained in $\mathcal W^s(\omega)$. Additionally, from (ii), it follows directly that $\mathcal M(\omega)$ is diffeomorphic to the unit circle, and, hence, we can conclude statement (iv).
\end{proof}

\section{Random isochrons} \label{sec:random_isochrons}
\subsection{Isochrons as stable manifolds} \label{sec:stable_manifold}

\subsubsection{Definition of forward isochrons}
Let $A$ be an attracting random cycle for the random dynamical system $(\theta, \varphi)$ where $A$ is a random forward attractor (and possibly also a random pullback attractor). One may think of equations of the type~\eqref{eq:exzz2_noise}, \eqref{Polar_original} or similar such that almost sure forward and pullback convergence coincide (see e.g.~\cite[Proof of Theorem B]{delr17} or \cite[Example 2.7 (i)]{Scheutzow02}). 
We further assume that we are in the situation of a differentiable hyperbolic random dynamical system as discussed in Section~\ref{sec:differentiability}.

In the typical setting of attracting random cycles, we may assume that $\lambda_1 = 0$ with single multiplicity and $\lambda_i < 0$ for all $2\leq i \leq p$.
In analogy to the stable manifolds of points on a deterministic limit cycle, we can then establish the following key novel definition (see also Figure~\ref{fig:random_stable_man}).
\begin{defn} \label{def:forwardisochron}
The \emph{random forward isochron} $W^{\txtf}(\omega, x)$ of a pair $(\omega, x) \in \Omega \times \R^m$ with $x \in A(\omega)$ is given by the stable set
\be
\label{eq:isochron_forward}
W^{\txtf}(\omega, x):= \left\{y\in \R^m:\lim_{t\ra +\I}\txtd(\varphi(t, \omega, y),\varphi(t, \omega, x))=0\right\},
\ee  
for almost all $\omega \in \Omega$ and all $x \in A(\omega)$.  In particular, we have for all $\tilde \lambda \in (0, - \lambda_2)$, where $\lambda_2$ denotes the largest nonzero Lyapunov exponent,
\be
\label{eq:isochron_forward_hyperbolic}
W^{\txtf}(\omega, x) =\left\{y\in \R^m:\sup_{t\geq 0} \txte^{\tilde \lambda t} \txtd(\varphi(t, \omega, y),\varphi(t, \omega, x)) < \infty \right\}.
\ee 
\end{defn}

\begin{remark}
It is clear from the definition why we exclude the case $\lambda_1 <0$. In this situation, the set $W^{\txtf}(\omega,x)$ is the whole absorbing set and, hence, no information about the decomposition of the state space by the dynamics can be obtained that way.

As indicated in Section~\ref{sec:chaotic_attractors}, a chaotic random attractor, characterized by $\lambda_1 >0$, also exhibits recurrence properties such that Definition~\ref{def:forwardisochron} can principally be also applied to this situation. 
However, it is arguable to what extent one can speak of isochronicity, given the irregular recurrence properties. Since this already makes isochronicity a difficult issue for deterministic chaotic oscillators \cite{Schwabedaletal2012}, we leave a detailed analysis of random isochrons for chaotic random attractors as a topic for future work.
\end{remark}
It is easy to observe that for all $s \geq 0$ we have
\be \label{Wfinvariance}
\varphi(s,\omega) W^{\txtf}(\omega, x) = W^{\txtf}(\theta_s \omega, \varphi(s,\omega,x)),
\ee
i.e.~the forward isochrons are $\varphi$-invariant, as depicted in Figure~\ref{fig:random_stable_man}.
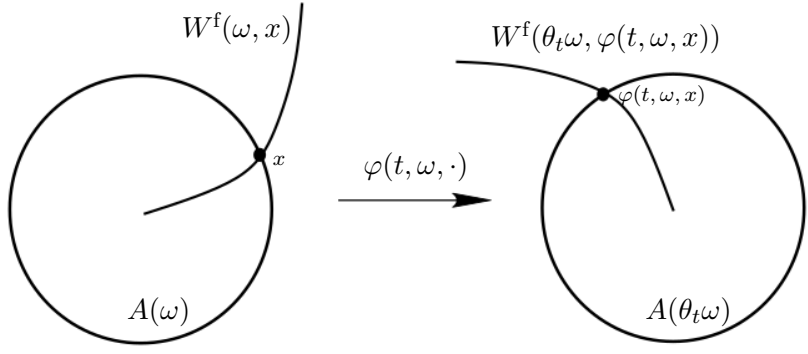
\begin{figure}[htbp]
        \centering
  		\begin{overpic}[width=0.7\textwidth]{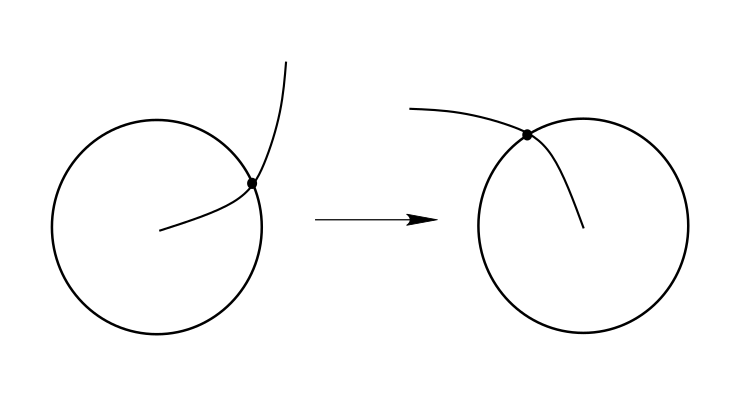}   
  		         \put(20,11){\small $A(\omega)$ }
  		         \put(77,11){\small $A(\theta_{t} \omega)$ }
                \put(46,27){\small $\varphi(t,\omega, \cdot) $}
                \put(36,28){\scriptsize $x$}
                \put(74,35){\scriptsize $\varphi(t,\omega, x) $}
                \put(26,42){\small $W^{\txtf}(\omega, x)$ }
                \put(60,41){\small $W^{\txtf}(\theta_t \omega, \varphi(t,\omega,x))$ }
        \end{overpic}
\caption{Sketch of isochrons $W^{\txtf}(\omega, x)$ at $A(\omega)$ and $W^{\txtf}(\theta_t \omega, \varphi(t,\omega,x))$ at $A(\theta_t \omega)$ as an illustration of Definition~\ref{def:forwardisochron} and the invariance relation~\eqref{Wfinvariance}, for $A(\omega)$ being a random limit cycle.}
        \label{fig:random_stable_man}
\end{figure}
\subsubsection{Existence and properties of random stable sets}
In the literature on (global) random dynamical systems, the existence of stable sets such as $W^{\txtf}(\omega,x)$ as stable manifolds is often first established for discrete time, see e.g.~\cite{r79} or \cite[Chapter III]{lq95}. (Arnolds treatment \cite[Chapter 7]{a98} is limited to equilibria.) Even though the local view in~\cite{LiLuBates}, as described in Section~\ref{sec:normally_hyperbolic}, is different, we need to also account for the global situation in order to provide the full picture.
Hence, we begin with adopting the discrete-time approach by reducing the analysis to time-one maps $\varphi(1,\omega, \cdot)$ and its concatenations
\begin{equation} \label{discreteRDS}
 \varphi(n, \omega, x) = (\varphi(1, \theta_{n-1} \omega, \cdot)\circ \varphi(1, \theta_{n-2} \omega, \cdot) \circ \cdots \circ \varphi(1, \omega, \cdot))(x), \ n \in \mathbb{N}  \,.
\end{equation}
First we want to conclude for all $\tilde \lambda \in (0, - \lambda_2)$ that
\be
\label{eq:isochron_forward_discrete}
 \tilde W^{\txts}(\omega, x):=\left\{y\in \mathbb{R}^m: \sup_{n\geq 0} e^{\tilde \lambda n} \txtd(\varphi(n, \omega, y),\varphi(n, \omega, x)) < \infty \right\}
\ee 
is an $(m-1)$-dimensional immersed $C^k$-submanifold under sufficient boundedness assumptions which would be immediately satisfied if the state space $\cX$ is a compact manifold (cf.~\cite[Chapter III, Theorem 3.2]{lq95}). We will state such conditions for our setting $\cX = \mathbb{R}^m$ in the following. The transition to the time-continuous case, i.e.~establishing $W^{\txtf}(\omega, x) = \tilde W^{\txts}(\omega, x)$, then follows immediately from the integrability assumption~\eqref{integrabilityassumpt} for the MET, as one can observe with the proof of \cite[Chapter V, Theorem 2.2] {lq95}.

One possible approach can be found in \cite{b12}: consider the maps~\eqref{discreteRDS}. For $x \in \mathbb{R}^m$, we define the local linear shift function
$$ f_x \,: \, \mathbb{R}^m \cong \txtT_x \mathbb{R}^m \to \mathbb{R}^m, \quad y \mapsto f_x(y) := x + y \,. $$
Further, we define the map
$$ F_{(\omega,x),n} \,:\, \txtT_{\varphi(n, \omega,x)} \mathbb{R}^m \to \txtT_{\varphi(n+1, \omega,x)} \mathbb{R}^m; \quad F_{(\omega,x),n} := f_{\varphi(n+1, \omega,x)} ^{-1} \, \circ \varphi(1, \theta_{n} \omega, \cdot) \circ f_{\varphi(n, \omega,x)}\,,$$
which is the evolution process of the linearization around the trajectory starting at $x \in \mathbb{R}^m$. Assume that there is an invariant probability measure $\mathbb{P} \times \rho$ for $(\Theta_t)_{t \geq 0}$ on $(\Omega \times \R^m, \mathcal{F}_0^\infty \times \mathcal{B}(\R^m))$ (see Appendix~\ref{sec:RDSSDE} and~\ref{sec:invariant_measures}). If the RDS is induced by an SDE, the measure $\rho$ is exactly the stationary measure of the associated Markov process. The integrability condition of the MET with respect to this measure reads
\be \label{ass:met} 
\log^+ \| \rmD \varphi(1, \omega, \cdot) \|  \in L^1(\mathbb{P} \times \rho)\,.
\ee
The crucial boundedness assumption that compensates for the lack of compactness in the proof of a stable manifold theorem reads
\be
\label{ass:secder}
\log \left( \sup_{ \xi \in B_1(x)} \| \rmD_{\xi}^2 F_{(\omega,x),0}\| \right) \in  L^1(\mathbb{P} \times \rho) \,,
\ee
where $\rmD^2$ is the second derivative operator and $B_1(x)$ denotes the ball of radius $1$ centered at $x \in \mathbb{R}^m$.

In the situation where the maps~\eqref{discreteRDS} of the discrete-time RDS are the time-one maps of the continuous-time RDS induced by the SDE~\eqref{SDEgen} with the stationary distribution fulfilling
\be \label{intlog0}
 \int_{\mathbb{R}^m} \log(\left\| x \right\| +1)^{1/2} \, \rmd \rho(x) < \infty\,,
\ee
we have the following requirements on $b, \sigma_i \in C^{k+1}$, $1 \leq i \leq n, k \geq 2$, such that 
assumption~\eqref{ass:secder} is satisfied:
\be \label{condition:Biskamp}
\|b\|_{k,\delta} + \sum_{i=1}^n \|\sigma_i\|_{k,\delta} < \infty\,,
\ee
where $0 < \delta \leq 1$ and with multi index notation $\alpha = (\alpha_1, \dots, \alpha_m)$, $\left|\alpha \right| = \sum_{i=1}^m \left|\alpha_i \right|$, for $f \in C^k$
\be \label{equ:norm}
\|f\|_{k,\delta} = \sup_{x \in \R^m} \frac{\|f(x)\|}{1 + \|x\|} + \sum_{1 \leq \left| \alpha \right| \leq k} \sup_{x \in \R^m} \| \rmD^{\alpha} f(x) \| + \sum_{\left| \alpha \right| = k}  \sup_{x \neq y} \frac{\| \rmD^{\alpha} f(x) -  \rmD^{\alpha} f(y)\| }{\|x-y\|^{\delta}}.
\ee
This means that the coefficients of the SDE have at most linear growth, globally bounded derivatives and the $k$-th derivatives have bounded $\delta$-H\"older norm. In \cite{b12}, also the backward flow and a condition similar to~\eqref{ass:secder} for the inverse are considered, but these are not needed when we purely regard the stable manifold problem. These conditions on the drift $b$ are generally too restrictive since already examples~\eqref{eq:exHopf},~\eqref{eq:exzz2} and~\eqref{eq:exzz2planar} are not covered. 
Of course, one can always consider the dynamics on a compact domain $\cK$, with absorbing or reflecting boundary conditions at the boundary of the domain, as will see later in Section~\ref{sec:connection_physics} for the averaged problem on the level of the Kolmogorov equations. However, this involves further technicalities for the random dynamical systems approach which we try to avoid here. The easiest way of reduction to a compact domain $\cK$ is to assume compact support of the noise and absorption to $\cK$ through the drift dynamics such that neither global nor boundary conditions are needed (see Theorem~\ref{thm:SMT} (iii)).

Additionally  we consider~\cite[Section 3]{fgs16} which discusses conditions for synchronization to a singleton random attractor for random dynamical systems induced by an SDE~\eqref{SDEgen} with additive noise, i.e.~$n=m$ and, for all $1\leq i,j \leq n$, $\sigma_i^j = \sigma \delta_{i,j}$ where $\sigma > 0$ and $\sigma_i^j$ denotes the $j$-th entry of the vector $\sigma_i$.
The authors formulate a special local stable manifold theorem for the case $\lambda_1 < 0$, which is, however, based on \cite{r79} where stable manifold theorems are considered in full generality. The assumption for deducing the local stable manifold theorem amounts to a (weaker) combination of conditions~\eqref{ass:met} and~\eqref{ass:secder}, and reads
\be \label{ass:fgs}
\mathbb{E} \int_{\mathbb{R}^m} \log^+ \| \varphi(1,\omega, \cdot + x) - \varphi(1,\omega,x) \|_{C^{1,\delta}(\overline B_1(0))} \, \rmd \rho(x) < \infty\,,
\ee
where $C^{1,\delta}$ is the space of $C^1$-functions whose derivatives are $\delta$-H\"{o}lder continuous for some $\delta \in (0,1)$ and $\rho$ denotes the stationary measure of the associated Markov process. We introduce a classical dissipativity condition, the \emph{one-sided Lipschitz condition}
\be \label{ass:dissipative}
\langle b(x) - b(y), x - y \rangle \leq \kappa \left\| x - y \right\|^2\,,
\ee
for all $x,y \in \mathbb{R}^m$ and $\kappa > 0$.
According to \cite[Lemma 3.9]{fgs16}, condition~\eqref{ass:fgs} is satisfied in the case of additive noise if $b \in C^2(\mathbb{R}^m)$ fulfills~\eqref{ass:dissipative}, admits at most polynomial growth of the second derivative, i.e.
\begin{equation} \label{secder}
\left\| \rmD^2 b(x) \right\| \leq C(\left\| x \right\|^M +1) \quad \text{for all } x \in \mathbb{R}^m \text{ and some } C>0, M \in \mathbb{N}\,, 
\end{equation}
and the stationary distribution $\rho$ satisfies
\be \label{intlog}
 \int_{\mathbb{R}^m} \log^+(\left\| x \right\|) \rmd \rho(x) < \infty\,.
\ee
\subsubsection{Main theorem about random isochrons}
Assumptions~\eqref{ass:dissipative} and~\eqref{secder} on the drift are weaker than condition~\eqref{condition:Biskamp} but, in \cite{fgs16}, only applied to situations with additive noise whereas at least linear multiplicative noise as in~\eqref{eq:exzz2} is a desirable model for random periodicity. We address this issue in Remark~\ref{rem:mainexample} and point (iii) of the following theorem, which summarizes the findings from above: 
\begin{thm}[Forward isochrons are stable manifolds]\label{thm:SMT}
Consider an ergodic $C^k$, $k \geq 2$, random dynamical system $(\theta, \varphi)$ on $R^m$ with random attractor $A$, satisfying the integrability assumption~\eqref{integrabilityassumpt} of the Multiplicative Ergodic Theorem such that $\lambda_1 = 0$ with single multiplicity and $\lambda_i < 0$ for all $2\leq i \leq p$. Let further one of the following assumptions be satisfied:
\begin{enumerate}[(i)]
\item The RDS $(\theta, \varphi)$ is induced by an SDE of the form~\eqref{SDEgen} such that the unique stationary measure $\rho$ satisfies~\eqref{intlog0} and the drift and diffusion coefficients satisfy~\eqref{condition:Biskamp},
\item The RDS $(\theta, \varphi)$ is induced by an SDE of the form~\eqref{SDEgen} with $n=m$ and, for all $1\leq i,j \leq n$, $\sigma_i^j = \sigma \delta_{i,j}$ where $\sigma > 0$,
such that the unique stationary measure $\rho$ satisfies~\eqref{intlog} and the drift satisfies conditions~\eqref{ass:dissipative} and~\eqref{secder},
\item The RDS $(\theta, \varphi)$ is induced by an SDE of the form~\eqref{SDEgen} such that $\supp(\sigma) \subset \mathbb{R}^m$ is compact, the drift $b$ satisfies condition~\eqref{ass:dissipative} with $\kappa < 0$ for all $ \|x\|, \|y\| > R$ for some $R>0$ and there is a unique stationary measure $\rho$ with $\supp(\rho) \subset \mathbb{R}^m$ compact.
\item the RDS satisfies the conditions of Proposition~\ref{prop:existence_of_random_cycles}.
\end{enumerate}
Then for almost all $\omega \in \Omega$ and all $x \in A(\omega)$ the random forward isochrons $W^{\txtf}(\omega, x)$ (see \eqref{eq:isochron_forward_hyperbolic})
are a uniquely determined $C^{k-1}$ in $x$ family of $C^k$  $(m-1)$-dimensional submanifolds (at least locally, i.e.~within a neighbourhood $\mathcal U$ of $x$) 
of the stable manifold $\mathcal W^s(\omega)$ such that
$$ \mathcal W^s(\omega) = \cup_{x \in A(\omega)} W^{\txtf}(\omega, x),$$
where the union is disjoint.
\end{thm}
\begin{proof}
As already discussed, in most of the cited literature, the stable manifold theorem is shown for discrete time. However, the transition to the time-continuous case, i.e.~establishing $W^{\txtf}(\omega, x) = \tilde W^{\txts}(\omega, x)$, follows immediately from the integrability assumption~\eqref{integrabilityassumpt} for the MET, as one can observe with the proof of \cite[Chapter V, Theorem 2.2] {lq95}.
Hence, the fact that the sets $W^{\txtf}(\omega, x)$ are a uniquely determined $C^{k-1}$ in $x$ family of $C^k$  $(m-1)$-dimensional submanifolds 
of the stable manifold $\mathcal W^s(\omega)$ can be deduced in various situations as follows:

Assumption (i) is derived from \cite[Theorem 4.7 and Theorem 9.1]{b12}, where $W^{\txtf}(\omega, x)$ are global stable manifolds.
Assumption (ii) is derived from \cite[Lemma 3.9]{fgs16} showing that the conditions for the local stable manifold theorem~\cite[Theorem 5.1]{r79} are satisfied, i.e.~$W^{\txtf}(\omega, x)$ is a $C^{k}$ submanifold of $\mathbb{R}^m$ of dimension $m-1$, at least within a neighbourhood $\mathcal U$ of $x$.
Furthermore, it is obvious from the assumptions that condition~\eqref{ass:fgs} is satisfied and, hence,  assumption (iii) is derived similarly to assumption (ii).
Assmuption (iv) can be taken according to \cite[Theorem 2.4]{LiLuBates}. 

This leaves to prove the foliation property in all these cases: the proof that
$$ \mathcal W^s(\omega) = \cup_{x \in A(\omega)} W^{\txtf}(\omega, x),$$
can be deducted in direct analogy to the proof of \cite[Proposition 9 (iv)]{LiLuBates}. The fact that the union is disjoint can be seen as follows: assume there is a $y \in W^{\txtf}(\omega, x) \cap W^{\txtf}(\omega, x')$ for $x\neq x'$. Since $A(\omega)$ is an invariant hyperbolic limit cycle and $x,x' \in A(\omega)$, we have that $d(\varphi(t, \omega,x), \varphi(t, \omega, x')) \geq \delta > 0$ for all $t \geq0$. Hence, we obtain by definition of $W^{\txtf}$ and the triangle inequality that
$$ 1 \leq \frac{d(\varphi(t, \omega,x), \varphi(t, \omega, y)) + d(\varphi(t, \omega,y), \varphi(t, \omega, x'))}{d(\varphi(t, \omega,x), \varphi(t, \omega, x'))} \to 0,$$
which is a contradiction (see proof of \cite[Proposition 9 (iii)]{LiLuBates} for a similar argument).
\end{proof}
\begin{remark} \label{rem:mainexample}
\begin{enumerate}
\item[(i)] One could also try to extend Theorem~\ref{thm:SMT} (ii) to the situation with any diffusion coefficients satisfying~\eqref{condition:Biskamp} instead of only additive noise. For showing this, first notice that under the assumptions on $\sigma$ the drift $\hat b= b + b_0$ with the It\^o-Stratonovich-conversion term
$$ b_0 := \frac{1}{2} \sum_{i=1}^n \sum_{j=1}^m \sigma_i^j \frac{\partial}{\partial x_j} \sigma_i$$
still satisfies assumptions~\eqref{ass:dissipative} and~\eqref{secder}.
Due to the mild behaviour~\eqref{condition:Biskamp} of the diffusion coefficients, one could then try to make analogous estimates as in \cite[Lemma 3.9]{fgs16} to induce that condition~\eqref{ass:fgs} is satisfied. Since we are mainly interested in the local behavior, we refrain from conducting such estimates here, but point out that this would be an interesting general extension.
\item[(ii)] Consider the example equation~\eqref{eq:exzz2planar} (and by that equation~\eqref{eq:exzz2}): the drift $b$ is polynomial such that condition~\eqref{secder} is satisfied and we have
\begin{align}\label{eq:b_estimate}
\langle b(x) -b(y), x-y\rangle &= \| x- y \|^2 - \|x\|^4 - \|y\|^4 + \langle x,y\rangle(\|x\|^2 + \|y\|^2)) \nonumber \\
&=  \| x- y \|^2 -  \|x\|^4 - \|y\|^4 +  \frac{1}{2}(\|x\|^2 + \|y\|^2)^2  \nonumber\\
&\quad - \frac{1}{2}(\|x\|^2 + \|y\|^2)\| x- y \|^2  \\
&= \left(1- \frac{1}{2}(\|x\|^2 + \|y\|^2)\right) \| x- y \|^2 - \frac{1}{2}(\|x\|^2 - \|y\|^2)^2  \nonumber \\
&\leq   \| x- y \|^2. \nonumber
\end{align}
Hence, also condition~\eqref{ass:dissipative} is satisfied. 
Furthermore, the unique stationary distribution $\rho$ has a density
\be \label{statdens:example}
p(x,y) = \frac{1}{Z} \left(x^2 + y^2\right)^{\frac{1}{\sigma^2}-1} \exp\left( - \frac{x^2 + y^2}{\sigma^2} \right),
\ee
solving the stationary Fokker-Planck equation.
Hence, also condition~\eqref{intlog} is fulfilled. Since the noise term is linear, we obviously have $\sum_{i=1}^n \|\sigma_i\|_{k,\delta} < \infty$ for all $k \geq 2, \delta \in (0,1]$. Hence, we could deduce the assertions of Theorem~\ref{thm:SMT} if we had proven the extension as discussed in (i).

However, for our purposes, this is not necessary:  we additionally have, using the same transformation as in estimate~\eqref{eq:b_estimate}, that for $R = \sqrt{3}$ and $\|x\|, \|y\| > R$
\begin{align*}
\langle b(x) -b(y), x-y\rangle &\leq 
\left(1- \frac{1}{2}(\|x\|^2 + \|y\|^2)\right) \| x- y \|^2 \\
&\leq  - \kappa \| x- y \|^2,
\end{align*}
where $\kappa = R^2/2 -1$. Now choosing a smooth cut-off of $\sigma$, say $\tilde \sigma$, such that $\sigma = \tilde \sigma$ on $B_{R^*}(0)$ for some large $R^* > R$ and $ \tilde \sigma \equiv 0$ on $\mathbb{R}^m \setminus B_{R^*+1}(0)$, we obtain a stationary density $\tilde p$ with $\tilde p = p/\tilde Z $  on $B_{R^*}(0)$, where $\tilde Z > 0$ is a normalization constant, and $ \tilde p \equiv 0$ on $\mathbb{R}^m \setminus B_{R^*+1}(0)$. Hence, we can apply Theorem~\ref{thm:SMT}~(iii).
In particular, note that this construction allows, when $\tilde \sigma$ is small enough, for a transformation into the random ODE~\eqref{eq:random_ODE} with sufficiently small bounded noise such that Proposition~\ref{prop:existence_of_random_cycles} and, by that, Theorem~\ref{thm:SMT}~(iv) can be applied. This procedure is, of course, independent from the particular form of equation~\eqref{eq:exzz2planar} and can be used for any SDEs around deterministic limit cycles when the transformation into the random ODE~\eqref{eq:random_ODE} is possible (which is always the case for additive and linear multiplicative noise).
\end{enumerate}
\end{remark}

Given~\eqref{eq:isochron_forward_hyperbolic}, we further assume that there exists a Crauel random periodic solution $(\psi,T)$ such that $\psi(t, \omega) \in A(\omega)$ for all $\omega \in \Omega$ and $t \geq 0$, as for example seen in Proposition~\ref{prop:crps_ex}. 
Then we can investigate the behaviour of
$$W^{\txtf}(\omega, \psi(0, \omega))=\left\{y\in \R^m:\lim_{t\ra +\I}\txtd(\varphi(t, \omega, y),\psi(t, \theta_t \omega))=0\right\}.$$
If, as in Proposition~\ref{prop:crps_ex}, each $x \in A(\omega)$ can be identified as $\psi_x(\omega,0)$ for some Crauel random periodic solution, then $T_x(\omega)$ is the period we can associate with $W^{\txtf}(\omega, \psi_x(0, \omega))$. We summarize this insight in the following definition:
\begin{defn}[Period of random forward isochron] \label{def:randomperiod}
Let $(\psi,T)$ be a Crauel random periodic solution for the RDS $(\theta, \varphi)$ such that $\psi(t, \omega) \in A(\omega)$ for all $\omega \in \Omega$ and $t \geq 0$, where $A$ is an attracting random cycle or chaotic random attractor. Then the we call $T(\omega)$ the period of the corresponding random forward isochron $W^{\txtf}(\omega, \psi(0, \omega))$ for all $\omega \in \Omega$. 
\end{defn}
The natural question arises whether
$$ \varphi(T_x(\omega), \omega, \cN_x(\omega)) \subset \cN_x(\theta_{T_x(\omega)} \omega)$$
holds for some measurable family $N_x(\omega)$ of cross-sections, in particular, whether we can identify $\cN_x(\omega)= W^{\txtf}(\omega, \psi_x(0, \omega))$. What we observe, is the following:

\begin{prop}
Let $(\theta, \varphi)$ be a random dynamical system with random attractor $A$ 
and the isochrons $W^{\txtf}(\omega, x)$ as given in \eqref{eq:isochron_forward} such that each $x \in A(\omega)$ can be identified with $\psi_x(0, \omega)$ for some Crauel random periodic solution $(\psi_x, T_x)$.
Then we have
\begin{equation} \label{invariance}
\varphi(T_x(\omega), \omega, W^{\txtf}(\omega, \psi_x(0, \omega))) \subset W^{\txtf}(\theta_{T_x(\omega)}\omega, \psi_x(T_x(\omega), \theta_{T_x(\omega)}\omega)).
\end{equation}
\end{prop}
\begin{proof}
Let $y \in  W^{\txtf}(\omega, \psi_x(0, \omega))$. Then
\begin{align*}
&\lim_{t\ra +\I}\txtd(\varphi(t, \theta_{T_x(\omega)}\omega, \varphi(T_x(\omega), \omega, y)),\psi_x(t + T_x(\omega), \theta_{t+T_x(\omega)} \omega)) \\
 &= \lim_{t\ra +\I}\txtd(\varphi(T_x(\omega) + t, \omega, y),\psi_x(t + T_x(\omega), \theta_{t+T_x(\omega)} \omega)) \\
  &= \lim_{s\ra +\I}\txtd(\varphi(s, \omega, y),\psi_x(s, \theta_{s} \omega)) =0.
\end{align*}
Hence, the statement follows directly.
\end{proof}
\subsubsection{Pullback isochrons}
In analogy to the different notions of a random attractor, one could also consider defining fiberwise isochrons for random dynamical systems in a pullback sense, as follows: 

Again assume there is a Crauel random periodic solution $(\psi,T)$ on an attracting random cycle $A$ (or chaotic random attractor $A$). Then the \emph{random pullback isochrons} could only be defined as 
\begin{align}
\label{eq:isochron_pullback}
 W^{\txtp}(\omega, \psi(0, \omega))&:=
  \left\{y\in \R^m:\lim_{t\ra +\I}\txtd(\varphi(t, \theta_{-t} \omega, y),\varphi(t, \theta_{-t} \omega, \psi(0, \theta_{-t} \omega))=0\right\} \nonumber \\
 &= \left\{y\in \R^m:\lim_{t\ra +\I}\txtd(\varphi(t, \theta_{-t} \omega, y),\psi(t, \omega))=0\right\},
\end{align} 
for almost all $\omega \in \Omega$. 

In contrast to the random forward isochron $W^{\txtf}(\omega, \psi(0, \omega))$, the set $W^{\txtp}(\omega, \psi(0, \omega))$ is not given as a stable set for the point $\psi(0, \omega)$ but as the set of points whose pullback trajectories converge to the trajectories starting in $\psi(0, \theta_{-t} \omega)$ as $t \to \infty$. Hence, such a definition cannot coincide with a stable manifold for a given point on a given fiber of the random attractor and, in particular, there does not seem to be a way to connect the set $W^{\txtp}(\omega, \psi(0, \omega))$ to the set $W^{\txtf}(\omega, \psi(0, \omega))$. In other words, it is not clear what geometric interpretation such a random pullback isochron could have and it is apparent that the definition in forward time, i.e.~Definition~\ref{def:forwardisochron}, yields the most immediately meaningful object in this context.

\subsection{The random isochron map} 
\label{sec:randomisomap}
For the following, recall the stochastic differential equation~\eqref{SDEgen} as
\begin{equation} \label{SDEgen_invariants}
\rmd X_t = b(X_t) \rmd t + \sum_{i=1}^n \sigma_i(X_t) \circ \rmd W_t^i\,,\qquad X_0=x,
\end{equation} 
where $W_t^i$ are independent real valued Brownian motions, $b$ is a $C^k$ vector field, $k \geq 1$, and $\sigma_1, \dots, \sigma_n$ are $C^{k+1}$ vector fields satisfying bounded growth conditions, as e.g. (global) Lipschitz continuity, in all derivatives to guarantee the existence of a (global) random dynamical system with cocycle $\varphi$ and derivative cocycle $\rmD \varphi$.

\begin{ex}
As before, the main examples we have in mind are two-dimensional. In particular, we may consider the corresponding
stochastic differential equation in 
polar coordinates $(\vartheta, r) \in [0, 2 \pi) \times [0, \infty)$
\be
\label{eq:polargeneral}
\begin{array}{lcl}
\rmd \vartheta &=& f_1(\vartheta, r) \, \rmd t + \sigma_1 g_1(\vartheta, r) \circ \, \rmd W_t^1,\\
\rmd r &=&  f_2(\vartheta, r) \, \rmd t + \sigma_2 g_2(\vartheta, r) \circ \, \rmd W_t^2\,.
\end{array}
\ee
As in Examples~\ref{ex:nonfixedphase} and~\ref{ex:chaos}, we usually regard a situation
such that in the deterministic case $\sigma_1 = \sigma_2 = 0$ there is an attracting limit cycle at $r=r^* >0$.
\end{ex}

From Theorem~\ref{thm:isochonsstableman_det} recall the isochron map $\xi: W^\txts(\gamma) \to \mathbb{R} \mod \tau_{\gamma} $ for a limit cycle $\gamma$ with period $\tau_{\gamma}$, which is given for every $y \in W^\txts(\gamma)$ as the unique $t$ such that $y \in W^\txts(\gamma(t))$, i.e.
\begin{equation*}
  \lim_{s\ra +\I}\txtd(\Phi(\gamma(\xi(y)),s),\Phi(y,s))=\lim_{s\ra +\I}\txtd(\gamma(s + \xi(y)),\Phi(y,s))=0 \,.
\end{equation*}
Analogously, we now introduce the following new notion for the random case; recall that for a CRPS $(\psi, T)$ we have, in particular, that $\psi(0, \omega) = \psi(T(\omega), \omega)$ for all $\omega \in \Omega$. 

\begin{thm} \label{def:randomisochronmap}
Consider the SDE~\eqref{SDEgen_invariants} such that the induced RDS has a random attractor $A$ with CRPS $(\psi, T)$ and parametrization $A(\omega) = \{ \psi(t+s, \omega)\,: \, t \in [0, T(\theta_{-s} \omega))\}$ for all $\omega \in \Omega$, $s \in \mathbb{R}$.
Then
\begin{enumerate}
\item there exists the \emph{random isochron map} $\tilde \phi$, i.e.~a measurable function $\tilde \phi: \mathbb{R}^m \times \Omega \times \mathbb{R} \to \mathbb{R}$, $C^k$ in the phase space variable, such that in a neighbourhood $\mathcal U(\omega)$ of $ A(\omega)$  we have
$$ \tilde\phi(\cdot, \omega, s): \mathcal U(\omega)\to [s,  s + T(\theta_{-s} \omega))$$
and for each $y \in \mathcal U(\omega)$, $s \in \mathbb{R}$
\begin{align} \label{eq:random_isochron_CRPS}
&\lim_{t\ra +\I}\txtd(\varphi(t, \omega, y),\varphi(t, \omega, \psi(\tilde \phi (y, \omega,s) , \omega))) \nonumber \\
& = \lim_{t\ra +\I}\txtd(\varphi(t, \omega, y),\psi(t + \tilde \phi(y, \omega,s), \theta_t \omega))=0,
\end{align}
\item for any $\omega \in \Omega $, $s \in \mathbb{R}$ and $t \in [0, T(\theta_{-s} \omega))$, the \emph{random $\tilde \phi$-isochron} $\tilde I(\omega, \psi(t+s, \omega),s)$ given by
\begin{equation} \label{defRDSiso_x_1}
\tilde I(\omega,\psi(t+s, \omega),s) = \{ y \in \mathcal{U}(\omega) : \tilde \phi( y, \omega,s) = \tilde \phi(\psi(t+s, \omega), \omega,s)\}
\end{equation} 
satisfies
\be \label{eq:random_forward_iso}
 \tilde I(\omega, \psi(t+s, \omega),s) = W^{\txtf}( \omega, \psi(t+s,  \omega)).
 \ee
\item for any $\omega \in \Omega $, $s \in \mathbb{R}$ and $y \in \mathcal{U}(\omega)$
\be \label{eq:invariance1}
\tilde \phi ( \varphi(s, \omega, y), \theta_s \omega, s) = \tilde \phi(y, \omega,0) +s\,,
\ee
or, equivalently,
\be \label{eq:invariance_derivative}
\frac{\rmd }{\rmd s} \tilde \phi ( \varphi(s, \omega, y), \theta_s \omega, s) = 1\,.
\ee
\end{enumerate}
\end{thm}

\begin{proof}
Since $A(\omega)$ is a random attractor, we have that for given $y \in \mathcal{U}(\omega)$ there is an $x \in A(\omega)$ such that $y \in W^{\txtf}(\omega, x)$. Due to the assumptions, for any $s \in \mathbb{R}$ there is a $t_x \in [0, T(\theta_{-s}\omega))$ such that $x = \psi(s + t_x, \omega)$. Then $\tilde \phi(y, \omega,s) := t_x + s$ satisfies the required properties, where measurability follows from the measurability of $T$ and, writing $t=t_x$, differentiability from
$$ \tilde I(\omega, \psi(t+s, \omega),s) = W^{\txtf}(\omega, \psi(t+s, \omega)),$$
which can be deduced as follows: we have $y \in \tilde I(\omega, \psi(t+s, \omega),s) $ if and only if
$\tilde \phi(y, \omega,s) = \tilde \phi(\psi(t+s, \omega), \omega,s) =t + s$ which, according to equation~\eqref{eq:random_isochron_CRPS}, is equivalent to
$$ \lim_{r\ra +\I}\txtd(\varphi(r, \omega, y),\varphi(r, \omega, \psi(t +s, \omega)))=0,$$
which is the case if and only if $y \in  W^{\txtf}(\omega, \psi(t+s, \omega))$.

It remains to show the third point: firstly, we derive from the invariance of the stable manifolds and equality~\eqref{eq:random_forward_iso} that
\begin{align} \label{eq:random_iso_invar}
&\varphi(s, \omega, \cdot) \tilde I (\omega, \psi(t, \omega), 0) = \varphi(s, \omega, \cdot) W^{\txtf}(\omega, \psi(t, \omega)) \nonumber\\
&= W^{\txtf}(\theta_s \omega, \psi(t+s, \theta_s \omega)) =  \tilde I (\theta_s \omega, \psi(t+s, \theta_s \omega), s)\,.
\end{align}
This means that for $x \in \mathcal{U}(\theta_s \omega)$ we have that $x = \varphi(s,\omega,y)$ for some $y \in \mathcal{U}( \omega)$ with $\tilde \phi (y, \omega,0) = t \in [0, T(\omega))$ if and only if 
$$ \tilde \phi (x, \theta_s \omega, s) = \tilde \phi (\psi(t+s, \theta_s \omega), \theta_s \omega, s) = t+s \,.$$
Hence, we obtain equation~\eqref{eq:invariance1}, or equivalently equation~\eqref{eq:invariance_derivative}, for any $y \in \mathcal{U}( \omega)$.
\end{proof}
Note that, due to the time dependence, we always give the image of the random isochron map $\tilde \phi(\cdot, \omega, s)$ as an interval $[s, s + T(\theta_{-s} \omega))$, in distinction from the deterministic case where the values of the isochron map $\xi$ lie in $\mathbb{R} \mod \tau_{\gamma}$, which can be identified with $[0, \tau_{\gamma})$, for fixed period $\tau_{\gamma}$ (see Proposition~\ref{prop:characterize}).
We are adding a couple of further remarks to the last theorem in order to highlight its coherence with the above and the analogy to the deterministic case.
\begin{remark} \label{rem:randomisochrons}
\begin{enumerate}
\item[(i)] As seen in the proof of Theorem~\ref{def:randomisochronmap}, note that for all $s \in \mathbb{R}$, $ t \in [0, T(\theta_{-s} \omega))$
\be \label{eq:phi_psi}
\tilde \phi(\psi(t+s, \omega), \omega, s) = t + s,
\ee
and, in particular, 
\be \label{eq:rewritephi}
\tilde \phi( \varphi(t, \theta_{-t} \omega, \psi(0,\theta_{-t} \omega)), \theta_t (\theta_{-t} \omega), 0) = \tilde \phi (\psi(t, \omega), \omega, 0) = t \fa t \in [0, T(\omega)).
\ee
Additionally, observe that the parametrization of the random attractor in Theorem~\ref{def:randomisochronmap} is generally possible when there is a CRPS; with Definition~\ref{def:CRPS} we have for all $s \geq 0$ that  $\psi(s + T(\omega), \theta_s \omega) = \psi(s, \theta_s \omega)$ and, hence, we can also consider 
$$A(\theta_s \omega) = \{ \psi(t + s, \theta_s \omega)\,: \, t \in [0, T(\omega))\},$$
for which we find, for $t \in [0, T(\omega))$,
$$ \tilde\phi(\cdot, \theta_s \omega, s): \mathcal U(\theta_s \omega)\to [s, s+ T(\omega)),  \ \tilde \phi (\psi(t+s, \theta_s \omega), \theta_s \omega, s) = t+s.$$
\item[(ii)]
From Proposition~\ref{prop:characterize} recall that the isochron map $\xi: W^\txts(\gamma) \to \mathbb{R} \mod \tau_{\gamma} $ for a deterministic limit cycle $\gamma$ satisfies equation~\eqref{eq:invariance_det}
$$
\frac{\txtd}{\txtd t} \xi (\Phi(y,t)) = 1 \ \text{ for all } t \geq 0, \, y \in W^\txts(\gamma)\,.
$$
Equation~\eqref{eq:invariance_derivative} is the analogous equation for the random dynamical system.
\item[(iii)]
In certain cases, it may be convenient to anchor the random $\tilde \phi$-isochrons at the deterministic limit cycle to compare with the averaging approaches from the physics literature later on.
Consider for example the SDE~\eqref{eq:polargeneral} with attracting limit cycle at $r=r^* >0$ in the deterministic case $\sigma_1 = \sigma_2 = 0$. 
We can then write the random isochron map $\tilde \phi: [0, 2 \pi) \times [0, \infty) \times \Omega \times \mathbb{R} \to \mathbb{R}$ such that in a neighbourhood $\mathcal U$ of the circle $\{r=r^*\}$ we have
$$ \tilde\phi(\cdot, \omega, s): \mathcal U \to [s,  s + T(\theta_{-s} \omega))$$
and, based on equations~\eqref{eq:invariance_derivative} and~\eqref{eq:invariance1},
\be \label{eq:invariance2}
\tilde \phi ( \varphi(s, \omega, (\vartheta_0, r_0)), \theta_s \omega, s) = \tilde \phi((\vartheta_0, r_0), \omega,0) +s\,,
\ee
or equivalently
\begin{equation} \label{dephasedefRDS2}
\rmd\, \tilde \phi(\varphi(s, \omega,(\vartheta_0, r_0)), \theta_s \omega, s) = 1 \, \rmd s\,,
\end{equation} 
for any $(\vartheta_0, r_0) \in \cU$, $s \in \mathbb{R}$ and $\omega \in \Omega$.
For any $\vartheta \in [0, 2 \pi)$, $s \in \mathbb{R}$ and $\omega \in \Omega$, we can write $\tilde I_{\vartheta}(\omega,s)$ for the level set
\begin{equation} \label{defRDSiso_2}
\tilde I_{\vartheta}(\omega,s) = \{ (\tilde \vartheta, \tilde r) \in \cU : \tilde \phi( (\tilde \vartheta, \tilde r), \omega, s) = \tilde \phi((\vartheta, r^*), \omega, s)\}.
\end{equation} 
\end{enumerate}
\end{remark}

Following Theorem~\ref{def:randomisochronmap}, we can simply define isochrons for any point $ x \in \mathcal U(\omega)$ by setting
\begin{equation} \label{defRDSiso_x_2}
\tilde I(\omega,x,s) := \tilde I(\omega,\psi(t+s, \omega),s) \ \text{ for } x \in \tilde I(\omega,\psi(t+s, \omega),s), t \in [0, T(\theta_{-s} \omega))\,.
\end{equation} 

We can then show the invariance of $\tilde I(\omega, x, 0)$ under the RDS, similarly to the invariance property~\eqref{Wfinvariance} of the forward isochrons, extending property~\eqref{eq:random_iso_invar} to any $x \in \mathcal U(\omega)$.
\begin{prop} \label{prop:inv_random_iso}
The random $\tilde \phi$-isochrons $\tilde I(\omega, x, 0)$ for $x \in \mathcal{U}(\omega)$ where $\mathcal{U}(\omega)$ is an attracting neighbourhood of $ A(\omega)$, are forward-invariant under the RDS cocycle, i.e.
\be  \label{eq:randisoinv}
\varphi(s,\omega) \tilde I(\omega, x, 0) \subset \tilde I (\theta_s \omega, \varphi(s,\omega,x), s) \ \text{ for almost all } \omega \in \Omega  \text{ and all } x \in \cU(\omega), s \geq 0.
\ee
\end{prop}
\begin{proof}
Let $y \in \varphi(s,\omega) \tilde I( \omega, x, 0)$. This means that there is a $z \in \mathbb{R}^m$ such that $y = \varphi(s, \omega,z)$ and $\tilde \phi (z, \omega, 0) = \tilde \phi (x, \omega, 0)$. We obtain from equation~\eqref{eq:invariance1} that
\begin{align*}
\tilde \phi (y, \theta_s \omega, s) &= \tilde \phi (\varphi(s, \omega,z), \theta_s \omega, s) \\
&= \tilde \phi(z, \omega, 0) + s \\
&=\tilde \phi(x, \omega, 0) + s\\
&= \tilde \phi (\varphi(s, \omega,x), \theta_s \omega, s).
\end{align*}
Hence, we have $y \in  \tilde I (\theta_s \omega, \varphi(s,\omega,x), s)$ and therefore 
$$ \varphi(s,\omega) \tilde I(\omega, x,0) \subset \tilde I (\theta_s \omega, \varphi(s,\omega,x),s).$$
This finishes the proof.
\end{proof}

\subsection{Stochastic isochrons via mean return time and random isochrons}
\label{sec:connection_physics}

One main approach to define stochastic
isochrons in the physics literature is due to Schwabedal and Pikovsky \cite{SchwabedalPikovsky} who introduce isochrons (or isophase surfaces) for noisy systems as sections $W^\E(x)$ with the mean first return time to the same section $W^\E(x)$ being a constant $\bar T$, equaling the average oscillation period. Note that such an object is not well-defined a priori as it seems unclear, what we imply here by ``return'', i.e., return to what?
The paper does not rigorously establish these objects but only gives a numerical algorithm which is successfully tested at the hand of several examples. 
According to the algorithm, a deterministic starting section $\cN$ is adjusted according to the mean return time, i.e., points are moved correspondent to the mismatch of their return time and the mean period for $\cN$, and this procedure is repeated until all points have the same mean return time. 
\subsubsection{The modified Andronov-Vitt-Pontryagin formula in \cite{caolindnerthomas19}}
\label{sec:caoetal}
Cao, Lindner and Thomas \cite{caolindnerthomas19} have made this approach rigorous by using a modified version of the Andronov-Vitt-Pontryagin formula for the mean first passage time (MFPT) $\tau_D$ on a bounded domain $D$ through its boundary $\partial D$. In more detail (cf.~\cite[Chapter 4.4]{schuss10}), the associated boundary value problem for $\mathcal L$ denoting the generator of the process, also called \emph{backward Kolmogorov operator}, is given by
\begin{equation} \label{Dynkin}
\mathcal{L} u(x) = -1 \fa x \in D, \quad u(x) = 0 \fa x \in \partial D,
\end{equation}
which is solved by 
$$ u(x) = \mathbb{E}[\tau_D | x(0) =x].$$ The problem in our case is that if we consider a domain whose absorbing boundary in $\theta$-direction is a line 
$\tilde l := \{(\tilde \vartheta(\tilde r), \tilde r) \,:\, R_1 \leq \tilde r \leq R_2 \}  $, where $\tilde \vartheta$ is a smooth function, the stochastic motion might not perform a full rotation to reach this boundary line. In particular, the mean return time for trajectories starting on $\tilde l$ will be zero. To circumvent this problem, Cao et al.~unwrap the phase by considering infinite copies of $\tilde l$ on the extended domain $ \mathbb{R} \times [R_1, R_2]$. For some $(\vartheta,r)$ with $\vartheta < 2 \pi < \tilde \vartheta (r)$, the mean first passage time $T(\vartheta, r)$ is then calculated  via the Andronov-Vitt-Pontryagin formula with periodic-plus-jump boundary condition in the $\vartheta$-direction and reflecting boundary condition in the $r$-direction.

In more detail, the process solving equation~\eqref{eq:polargeneral}, or its It\^o version respectively, with strongly elliptic generator $\mathcal L$ and its adjoint $\mathcal L^*$, the \emph{forward Kolmogorov operator}, is assumed to have a unique stationary density $\rho$ on $\Omega = [0, 2 \pi) \times [R_1, R_2]$ solving the stationary Fokker-Planck equation
$$ \mathcal{L}^* \rho = 0\,,$$
with reflecting (Neumann) boundary conditions at $r \in \{R_1, R_2\}$ and periodic boundaries $\rho(0,r) = \rho(2 \pi ,r)$ for all $r \in [R_1, R_2]$.
For model~\eqref{eq:polargeneral}, the stationary probability current $J_{\rho}$ reads, for $j=1,2$,
$$ J_{\rho, j} (\vartheta, r) = \left(f_j(\vartheta, r) + \frac{1}{2}g_j(\vartheta, r) \partial_j g_j(\vartheta, r) \right) \rho(\vartheta, r) - \frac{1}{2} \partial_j \left(g_j^2(\vartheta,r) \rho(\vartheta, r) \right). $$
Furthermore, for a $C^1$-function $\gamma: [R_1, R_2] \to [0, 2 \pi]$ the graph $C_{\gamma}$ (cf.~$\tilde l$ above) separates the domain $\Omega_{\textnormal{ext}}= \mathbb{R} \times [R_1, R_2]$ into a left and right connected component, with unit normal vector $n(r)$ oriented to the right. It is then assumed that the mean rightward probability flux through $C_{\gamma}$ is positive, which means that
\be \label{eq:probflux}
 \bar J_{\rho} := \int_{R_1}^{R_2} n^{\top}(r) J_{\rho} (\gamma(r),r)
 \, \rmd r > 0.
\ee
The mean period of the oscillator is then given as
\be \label{eq:meantime}
\bar T = \frac{1}{\bar J_{\rho}}\,.
\ee
The modified Andronov-Vitt-Pontryagin formula is then given by the following PDE, with reflecting and jump-periodic boundary conditions
\begin{align} \label{eq:Caoetal_PDE}
\mathcal L T &= -1, \quad \text{on } \Omega, \nonumber \\
g_2^2(\vartheta, r) \partial_2 T(\vartheta, r) &= 0, \quad \forall \vartheta \in \mathbb{R}, r \in \{R_1, R_2\}\,\\
T(\vartheta, r) - T(\vartheta + 2 \pi, r) &= \bar T, \quad \forall (\vartheta, r) \in \Omega_{\textnormal{ext}} \nonumber\,.
\end{align}
In fact, the last condition can be weakened to
\be \label{eq:milder_condition}
 T(0, r) - T(2 \pi, r) = \bar T, \quad \forall r \in [R_1,R_2].
\ee

Under the discussed assumptions, it is then shown in \cite[Theorem 3.1]{caolindnerthomas19} that the equation has a solution $T(\vartheta, r)$ on $\Omega_{\textnormal{ext}}$ and, hence by restriction, on $\Omega$, which is unique up to an additive constant. The level sets of $T(\vartheta,r)$ are then supposed to be the stochastic isochrons $W^\E((\vartheta,r))$ with mean return time $\bar T$ and associated isophase (up to some constant $\bar \Theta_0$)
$$ \bar \Theta(\vartheta, r) = - T(\vartheta, r) \frac{2 \pi}{\bar T}\,, $$
which therefore satisfies
\be \label{eq:BKstat_cao}
\mathcal{L} \bar \Theta = \frac{2 \pi}{\bar T}.
\ee
\subsubsection{Relation to random isochrons}
Recall from Definition~\ref{def:randomperiod} that, for a CRPS $(\psi,T)$, the random period $T(\omega)$ corresponds to the random forward isochron $W^{\txtf}(\omega, \psi(0, \omega))$ for all $\omega \in \Omega$. Hence, we can define the expected period as
\be \label{eq:expectedperiod_RDS}
\bar T_{\textnormal{RDS}} := \mathbb{E} [T(\cdot)]\,,
\ee 
where the 
index \emph{RDS} indicates the random dynamical systems perspective.
In the following, we discuss how $\bar T_{\textnormal{RDS}}$ is related to $\bar T$ and the isochron function $\bar \Theta$ \eqref{eq:BKstat_cao}.
\subsubsection*{Expectation of random period}
Similarly to Section~\ref{sec:caoetal}, consider equation~\eqref{eq:polargeneral} in an annulus $\mathcal R$ given by $0 \leq R_1 \leq r \leq R_2 \leq \infty$, i.e.~including the full space case $\mathcal R = [0, \infty) \times [0, 2 \pi)$.
Consider the slightly modified version of the PDE system~\eqref{eq:Caoetal_PDE}
\begin{align} \label{eq:expt_time_PDE}
\mathcal L u &= -1, \quad \text{on } (-2 \pi, 2 \pi) \times (R_1, R_2) , \nonumber \\
 u(\pm 2 \pi, r) -  u(0, r) &= \bar T, \quad \forall r \in (R_1,R_2)\,,
\end{align}
where for the case $R_1 > 0, R_2 < \infty$ one can take again Neumann boundary conditions 
$$g_2^2(\vartheta, r) \partial_2 u(\vartheta, r) = 0, \quad \forall \vartheta \in \mathbb{R}, r \in \{R_1, R_2\}.$$ 
Then we can formulate the following observation.
\begin{prop} \label{prop:expect_period}
Assume that system~\eqref{eq:polargeneral} has a CRPS $(\psi,T)$, fixing $\psi(0, \omega) \in \{0\} \times (R_1, R_2)$ for $0 \leq R_1 < R_2 \leq \infty$, where the RDS and its attracting random cycle are supported within $(R_1, R_2) \times [0, 2 \pi)$ (see Theorem~\ref{thm:SMT} and Remark~\ref{rem:mainexample}).
Then we obtain that
\begin{enumerate}
\item[(a)] the expectation of the random period is given by
\be \label{eq:expectation_T}
\mathbb{E} [T(\cdot)] =  \mathbb{E} [u(\psi(- T(\cdot),\theta_{-T(\cdot)}(\cdot)))] - \mathbb{E} [u(\psi(0,\cdot))],
\ee
where $u$ solves equation~\eqref{eq:expt_time_PDE},
\item[(b)] and, in particular, if the radial components of $\psi(0,\cdot)$ and $\psi(- T(\cdot),\theta_{-T(\cdot)}(\cdot))$ are equally distributed on $(R_1, R_2)$, we have
$$\bar T = T_{\textnormal{RDS}} = \mathbb{E} [T(\cdot)]\,.$$
\end{enumerate}
\end{prop}
\begin{proof}
As we have seen in the proof of Proposition~\ref{prop:crps_ex}, the period $T(\omega)$ has to satisfy for system~\eqref{eq:polargeneral}
$$T(\omega) = \inf \left\{ t > 0: \left| \int_{-t}^{0} f_1(\psi(s, \theta_s \omega)) \rmd s +  \sigma_1 \int_{-t}^{0} g_1(\psi(s,\theta_s \omega)) \circ \,  \rmd W_s^2(\omega) \right| = 2 \pi \right\}\,.$$
Hence, using Dynkin's equation for the solution $u$ of the boundary value problem~\eqref{eq:expt_time_PDE}, we obtain
\begin{align*}
\mathbb{E}[u(\psi(0,\cdot))] &= \mathbb{E}[u(\psi(-T(\cdot), \theta_{-T(\cdot)}(\cdot)))] + \mathbb{E}\left[ \int_{-T(\omega)}^0 L u(\psi(s, \theta_s \cdot))\, \rmd s \right]\\
&= \mathbb{E}[u(\psi(-T(\cdot), \theta_{-T(\cdot)}(\cdot)))] - \mathbb{E}[T(\cdot)],
\end{align*}
which shows claim (a).

Claim (b) follows straightforwardly, inserting~\eqref{eq:expt_time_PDE} into~\eqref{eq:expectation_T}.
\end{proof}
Note that this result is consistent with the basic Example~\ref{ex:basic_fixed_per}, where we have $T(\omega) = \bar T$ for all $\omega  \in \Omega$ since in this case $\psi(0,\cdot)$ and $\psi(- T(\cdot),\theta_{-T(\cdot)}(\cdot))$ are both distributed according to the stationary radial solution $r^*(\omega)$. 
Addtionally note that $u$ is the isochron function via mean return time, as discussed in Section~\ref{sec:caoetal}.

\subsubsection*{Expectation of isochron function}
Furthermore, we want to give an alternative derivation to Section~\ref{sec:caoetal} of an isochron function $\bar \phi((\vartheta,r)): \mathcal{R} \to \mathbb{R}$, yielding the sections $W^\E((\vartheta,r))$ with fixed mean return time given as level sets
\be \label{level_sets_SP}
W^\E((\vartheta,r)) = \{ (\tilde \vartheta,\tilde r) \in \mathcal{R}\,:\, \bar \phi((\tilde \vartheta,\tilde r)) = \bar \phi((\vartheta,r))  \}.
\ee
In more detail, we try to find the function $\bar \phi$ via an expected version of equations~\eqref{eq:invariance2} and~\eqref{dephasedefRDS2}. 
We fix $(\vartheta_0, r_0) \in \mathcal R$ and require that the function $\bar \phi$ satisfies along solutions $(\vartheta(t), r(t))$ of the SDE~\eqref{eq:polargeneral} the equality  (cf.~equation~\eqref{eq:invariance_det} in the deterministic case)
\begin{equation} \label{dephasedef}
\mathbb{E} \left[ \rmd\, \bar \phi(\vartheta(t),r(t)) | (\vartheta(0),r(0)) = (\vartheta_0, r_0)\right] = 1 \, \rmd t\,.
\end{equation} 
By this, we can show the following result:
\begin{prop}
There is   $\bar \phi((\vartheta,r)): \mathcal{R} \to \mathbb{R}$ and a period $\bar T > 0$ with
\begin{equation} \label{dephasedef2}
\mathbb{E} \left[  \bar \phi(\vartheta(t),r(t)) | (\vartheta(0),r(0)) = (\vartheta_0, r_0)\right] = \bar \phi(\vartheta(0), r(0)) + t \mod \bar T\,.
\end{equation} 
This $\bar T$ is the expected return time to the  isochron $W^\E((\vartheta_0,r_0))$ which is the level set of $\bar \phi(\vartheta_0, r_0)$.

In particular, the function $\bar \phi$ can be identified with the solution $\bar \Theta$ of equation~\eqref{eq:BKstat_cao}.
\end{prop}
\begin{proof}
Using the chain rule of Stratonovich calculus and inserting~\eqref{eq:polargeneral}, Equation~\eqref{dephasedef} can be rewritten as
\begin{align*}
1 \, \rmd t &= \mathbb{E} \left[ \frac{\rmd}{\rmd \vartheta}\, \bar \phi(\vartheta(t),r(t)) \, \rmd \vartheta + \frac{\rmd}{\rmd r}\, \bar \phi(\vartheta(t),r(t)) \, \rmd r \bigg| (\vartheta(0),r(0)) = (\vartheta_0, r_0)\right] \\
 &= \mathbb{E} \bigg[ \frac{\rmd}{\rmd \vartheta}\, \bar \phi(\vartheta(t),r(t)) \left(f_1(\vartheta(t), r(t)) \, \rmd t + \sigma_1 g_1(\vartheta(t), r(t)) \circ \, \rmd W_t^1 \right) \\
&+ \frac{\rmd}{\rmd r}\, \bar \phi(\vartheta(t),r(t)) \left(f_2(\vartheta(t), r(t)) \, \rmd t 
 + \sigma_2 g_2(\vartheta(t), r(t)) \circ \, \rmd W_t^2 \right) \bigg| (\vartheta(0),r(0)) = (\vartheta_0, r_0)\bigg]  \,,
\end{align*}
where the boundary condition in angular direction is
\be \label{eq:periodicbound} 
\bar \phi (2 \pi, r) =\bar \phi (0, r) \mod \bar T, 
\ee
for all $R_1 \leq r \leq R_2$, fixing
$$\bar \phi (0, r^*) = 0\,,$$
and
$$ \bar T = \bar \phi(2 \pi, r^*)\,.$$
In radial direction, if $0 < R_1 < R_2 < \infty$, one can choose reflecting boundary conditions as in Section~\ref{sec:caoetal}.

Writing time $t$ as an index, transforming the Stratonovich noise terms into It\^o noise terms and using the fact that the It\^o noise terms have zero expectation, leads to the equation
\begin{align} \label{eq:BKtime}
1 &= \mathbb{E} \bigg[ \left(f_1(\vartheta_t, r_t) + \frac{1}{2}g_1(\vartheta_t, r_t) \frac{\partial}{\partial \vartheta}g_1(\vartheta_t, r_t),  f_2(\vartheta_t, r_t) + \frac{1}{2}g_2(\vartheta_t, r_t) \frac{\partial}{\partial r}g_2(\vartheta_t, r_t) \right) \cdot \nabla \bar \phi (\vartheta_t, r_t) \nonumber\\ 
&+ \frac{1}{2} \sigma_1^2 g_1^2(\vartheta_t,r_t) \frac{\partial^2}{\partial \vartheta^2} \phi(\vartheta_t,r_t)  + \frac{1}{2} \sigma_2^2 g_2^2(\vartheta_t,r_t) \frac{\partial^2}{\partial r^2} \bar \phi(\vartheta_t,r_t) \bigg|  (\vartheta(0),r(0)) = (\vartheta_0, r_0))\bigg] \nonumber \\
&= \mathbb{E} \bigg[ \mathcal L \bar \phi(\vartheta_t,r_t)\bigg|  (\vartheta(0),r(0)) = (\vartheta_0, r_0)\bigg]\,,
\end{align}
where $\mathcal L$ denotes the backward Kolmogorov operator associated with the SDE~\eqref{eq:polargeneral}. In particular, a solution is given by the stationary version
\begin{equation} \label{eq:BKstat}
\mathcal L \bar \phi(\vartheta,r) = 1,
\end{equation}
with boundary condition~\eqref{eq:periodicbound}.
 Note that, up to the change of sign $\phi \to - \phi$, Equation~\eqref{eq:BKstat} is  Dynkin's equation and that Equation~\eqref{eq:BKstat_cao} is equivalent to equation~\eqref{eq:BKstat} with boundary condition~\eqref{eq:periodicbound} such that $\bar \phi$ is taken as a function from the domain $\Omega$ to $\mathbb{R} \mod \bar T$.
Hence, the two approaches, one starting with~\eqref{dephasedef} and the other, considering the MFPT, lead to the same outcome regarding the stochastic isochrons $W^\E((\vartheta,r))$.
\end{proof}


We exmplify this derivation of an isochron function $\bar \phi$ 
by reference to the fundamental  Example~\ref{ex:nonfixedphase}:
\begin{ex}
Recall Example~\ref{ex:nonfixedphase} with equation~\eqref{eq:exzz2_noise}, i.e.~in its most general form,
\begin{equation*}
\begin{array}{lcl}
\rmd \vartheta &=& h(r) \, \rmd t + \tilde h(r) \circ \, \rmd W_t^2,\\
\rmd r &=&  (r - r^3) \, \rmd t + \sigma r \circ \, \rmd W_t^1\,,
\end{array}
\end{equation*}
choosing $h(r) = \kappa + (r^2 -1) $, $\kappa \geq 1$, similarly to \cite[Example (1)]{SchwabedalPikovsky}, and $\tilde h$ some arbitrary smooth and bounded function. Note that $r^*=1$ for this case and that there is a stationary density $p$ for the radial process which has the form
$$ p(r) = \frac{1}{Z} r^{\frac{2}{\sigma^2} -1} \txte^{-\frac{r^2}{\sigma^2}} ,$$
where $Z > 0$ is a normalization constant. One can then additionally observe that $\mathbb{E}_{p} [r^2] = 1$ for all $\sigma \geq 0$, and, hence, $\mathbb{E}_{p} [h(r)] = \kappa.$

It is easy to see that 
$$ \hat \phi (\vartheta, r) = \frac{1}{\kappa}(\vartheta + \ln r)$$ 
solves~\eqref{eq:BKstat} such that \eqref{dephasedef2} is actually satisfied with $\bar T = \frac{2\pi}{\kappa}$. 
In fact, we have (up to some constant $\bar \phi_0$)
$$ \bar \phi (\vartheta, r) = \frac{1}{\kappa}(\vartheta + \ln r) \mod \bar T,$$ 
which, in this case, is also the deterministic isochron.
\end{ex}

Similarly to $T_{\textnormal{RDS}} := \mathbb{E} [T(\cdot)]$,
we can introduce for the associated random isochron map $\tilde \phi$  the expected quantity
\begin{equation} \label{eq:expected_isochronmap}
\bar \phi_{\textnormal{RDS}}(x) = \mathbb{E} [ \tilde \phi (x, \cdot, 0 )], 
\end{equation}
for fixed $x \in \mathbb{R}^m$, where $\tilde \phi$ is the random isochron map from Section~\ref{sec:randomisomap}.
It remains to clarify how the isochron function $\bar \phi$, or equivalently $\bar \Theta$ \eqref{eq:BKstat_cao},  may be related to $\bar \phi_{\textnormal{RDS}}$, assuming the existence of a CRPS $(\psi, T)$ as for Example~\ref{ex:nonfixedphase} (see Proposition~\ref{prop:crps_ex}). We give a brief discussion of a possible approach to this question in Appendix~\ref{appendix:rel_rim}, leaving a more thourough investigation as future work.

\section{Conclusion} \label{sec:concl}

We have introduced a new perspective on the problem of stochastic isochronicity, by considering random isochrons as random stable manifolds anchored at attracting random cycles with random periodic solutions. 
We have further characterized these random isochrons as level sets of a time-dependent random isochron map. 
Precisely this time-dependence of the random dynamical system, i.e.,~its non-autonomous nature, makes it difficult to specify the concrete relation to the definitions of stochastic isochrons given by fixed expected mean return times for whom we have given an alternative derivation of the isochron function $\bar \phi$ with return time $\bar T$.
We suggest an extended investigation of their relationship to the expected quantity $\bar \phi_{\textnormal{RDS}}$ as an intriguing problem for future work.
Additionally, it would be interesting to study the relation between stochastic isochronicity via eigenfunctions of the backward Kolmogorov operator  \cite{ThomasLindner} and \emph{random Koopman operators} (see \cite{Crnjaric-Zic2019}), extending the eigenfunction approach from the deterministic setting to the random dynamical systems case. 
\medskip

\textbf{Acknowledgments:} The authors gratefully acknowledge support by the DFG via the SFB/TR109 Discretization in Geometry and Dynamics. ME has also been supported by Germany's Excellence Strategy -- The Berlin Mathematics Research Center MATH+ (EXC-2046/1, project ID: 390685689). CK acknowledges support by a Lichtenberg Professorship of the VolkswagenFoundation.

\appendix 
\numberwithin{equation}{section}
\makeatletter 
\newcommand{\section@cntformat}{Appendix A}
\makeatother

\section{Random dynamical systems} \label{sec:app}

In this appendix we have collected several constructions for reference from the theory of random dynamical systems, which we have used throughout the main part of this work.

\subsection{Random dynamical systems induced by stochastic differential equations}
\label{sec:RDSSDE}

Following \cite{fgs16}, we make the following definition:
\begin{defn}[White noise RDS] \label{whitenoiseRDS}
Let $(\theta, \varphi)$ be a random dynamical system over a probability space $(\Omega,\mathcal F,\mathbb P)$ on a topological space $\cX$ where $\varphi$ is defined in forward time. 
Let $(\mathcal{F}_s^t)_{-\infty \leq s \leq t \leq \infty} $ be a family of sub-$\sigma$-algebras of $\mathcal F$ such that 
\begin{enumerate}[(i)]
\item $\mathcal F_t^u \subset \mathcal F_s^v$ for all $ s \leq t \leq u \leq v$,
\item $\mathcal F_s^t$ is independent from $\mathcal F_u^v$ for all $ s \leq t \leq u \leq v$,
\item $ \theta_r^{-1}(\mathcal F_s^t) = \mathcal{F}_{s+r}^{t+r}$ for all $ s \leq t$, $r \in \mathbb{R}$,
\item $\varphi(t, \cdot, x)$ is $\mathcal F_0^t$-measurable for all $t \geq 0$ and $x \in \cX$.
\end{enumerate}  
Furthermore we denote by $\mathcal{F}_{-\infty}^t$ the smallest $sigma$-algebra containing all $\mathcal F_s^t$, $s \leq t$, and by $\mathcal{F}_{t}^{\infty}$ the smallest $sigma$-algebra containing all $\mathcal F_t^u$, $t \leq u$. Then $(\theta, \varphi)$ is called a \emph{white noise (filtered) random dynamical system}.
\end{defn}
Consider a stochastic differential equation (SDE)
\begin{equation} \label{SDE_RDS}
\rmd X_t = f(X_t) \rmd t + g(X_t) d W_t, \ X_0 \in \mathbb{R}^d\,,
\end{equation} 
where $(W_t)$ denotes some r-dimensional standard Brownian motion, the drift $f: \mathbb{R}^d \to \mathbb{R}^d$ is a locally Lipschitz continuous vector field and the diffusion coefficient $g: \mathbb{R}^d \to \mathbb{R}^{d \times r}$ a Lipschitz continuous matrix-valued map.  If in addition $f$ satisfies a bounded growth condition, as for example a one-sided Lipschitz condition, then by \cite{ds11} there is a white noise random dynamical system $(\theta, \varphi)$ associated to the diffusion process solving~\eqref{SDE_RDS}. The probabilistic setting is as follows: We set $\Omega=C_0(\mathbb R,\mathbb R^r)$, i.e.~the space of all continuous functions $\omega:\mathbb R\rightarrow \mathbb R^r$ satisfying that $\omega(0)=0\in \mathbb R^r$. If we endow $\Omega$ with the compact open topology given by the complete metric
\[
\kappa(\omega,\widehat{\omega})
:=
\sum_{n=1}^{\infty}
\frac{1}{2^n}
\frac{\|\omega-\widehat{\omega}\|_n}{1+\|\omega-\widehat{\omega}\|_n},\quad
\|\omega-\widehat \omega\|_n:=\sup_{|t|\leq n } \|\omega(t)-\widehat{\omega}(t)\|\,,
\]
we can set $\mathcal F = \mathcal{B} (\Omega)$, the Borel-sigma algebra on $(\Omega,\kappa)$. There exists a probability measure $\mathbb P$ on $(\Omega,\mathcal F)$ called \emph{Wiener measure} such that the $r$ processes $(W_t^1), \dots, (W_t^r)$ defined by $(W_t^1(\omega), \dots, W_t^r(\omega))^{\mathrm T}:=\omega(t)$ for $\omega\in\Omega$ are independent one-dimensional Brownian motions. Furthermore, we define the sub-$\sigma$-algebra $\mathcal{F}_s^t$ as the $\sigma$-algebra generated by $\omega(u)- \omega(v)$ for $s \leq v \leq u \leq t$. The ergodic metric dynamical system $(\theta_t)_{t\in\mathbb R}$  on $(\Omega,\mathcal F,\mathbb P)$  is given by the shift maps
\begin{equation*}
\theta_t:\Omega\rightarrow \Omega, \quad (\theta_t \omega)(s) = \omega(s+t) - \omega(t)\,.
\end{equation*}
Indeed, these maps form an ergodic flow preserving the probability $\mathbb P$, see e.g. \cite{a98}.

Note that, by the It\^o-Stratonovich conversion formula, euqation~\eqref{SDE_RDS} with Stratonovich noise instead of It\^o noise also induces a random dynamical system under analogous assumptions.

\subsection{Invariant measures} \label{sec:invariant_measures}
Let $(\theta, \varphi)$ be a random dynamical system with the cocycle $\varphi$ being defined on one-or two-sided time $\mathbb T \in \{ \mathbb{R}_0^+, \mathbb{R} \}$. Then the system generates a skew product flow, i.e.~a family of maps $(\Theta_t)_{t \in \mathbb{T}}$ from $\Omega \times \cX $ to itself such that for all $t \in \mathbb T$ and $\omega \in \Omega, x \in \cX$
\begin{equation*}
\Theta_t(\omega, x) = (\theta_t \omega, \varphi(t, \omega,x))\,.
\end{equation*}
The notion of an invariant measure for the random dynamical system is given via the invariance with respect to the skew product flow, see e.g. \cite[Definition 1.4.1]{a98}. We denote by $ T\mu$ the push forward of a measure $\mu$ by a map $T$, i.e. $T \mu(\cdot) = \mu (T^{-1}(\cdot))$.
\begin{defn}[Invariant measure] A probability measure $\mu$ on $\Omega \times \cX$ is invariant for the random dynamical system $(\theta, \varphi)$ if
\begin{enumerate}
\item[(i)] $\Theta_t \mu = \mu$ for all $ t \in \mathbb T$\,,
\item[(ii)] the marginal of $\mu$ on $\Omega$ is $\mathbb{P}$, i.e.~$\mu$ can be factorised uniquely into $\mu(\rmd \omega, \rmd x) = \mu_{\omega}(\rmd x) \mathbb{P}(\rmd \omega)$ where $\omega \mapsto \mu_{\omega}$ is a \textit{random measure} (or \textit{disintegration} or \textit{sample measure}) on $\cX$, i.e.~$\mu_{\omega}$ is a probability measure on $\cX$ for $\mathbb{P}$-a.a. $\omega \in \Omega$ and $\omega \mapsto \mu_{\omega}(B)$ is measurable for all $B \in \mathcal{B}(\cX)$.
\end{enumerate}
\end{defn}
The marginal of $\mu$ on the probability space is demanded to be $\mathbb{P}$ since we assume the model of the noise to be fixed.
Note that the invariance of $\mu$ is equivalent to the invariance of the random measure $\omega \mapsto \mu_{\omega}$ on the state space $\cX$ in the sense that
\begin{equation} \label{Markovmeasure}
\varphi(t,\omega, \cdot) \mu_{\omega} = \mu_{\theta_t \omega} \quad \mathbb{P}\text{-a.s. for all} \ t \in \mathbb T\,.
\end{equation}
For white noise random dynamical systems $(\theta, \varphi)$, in particular random dynamical systems induced by a stochastic differential equation, there is a one-to-one correspondence between certain invariant random measures and stationary measures of the associated stochastic process, first observed in \cite{c91}. In more detail, we can define a Markov semigroup $(P_t)_{t \geq 0}$ by setting
$$ P_t f(x) = \mathbb{E}(f(\varphi(t, \cdot,x))$$
for all measurable and bounded functions $f: \cX \to \mathbb{R}$. If $\omega \mapsto \mu_{\omega}$ is a $\mathcal{F}_{-\infty}^{0}$-measurable invariant random measure in the sense of \eqref{Markovmeasure}, also called \textit{Markov measure}, then
\begin{equation*}
\rho (\cdot) = \mathbb{E} [\mu_{\omega}(\cdot)] = \int_{\Omega} \mu_{\omega} (\cdot) \mathbb{P}(d \omega)
\end{equation*}
turns out to be an invariant measure for the Markov semigroup $(P_t)_{t \geq 0}$, often also called stationary measure for the associated process. If $\rho$ is an invariant measure for the Markov semigroup, then
\begin{equation*}
\mu_{\omega} = \lim_{t \to \infty} \varphi(t, \theta_{-t}\omega, \cdot)\rho
\end{equation*}
exists $\mathbb{P}$-a.s.~and is an $\mathcal{F}_{-\infty}^{0}$-measurable invariant random measure.

We observe similarly to \cite{b91} that, in the situation of $\mu$ and $\rho$ corresponding in the way described above,
$$ \mathbb E [ \mu_{\omega}(\cdot) | \mathcal{F}_{0}^{\infty}] = \mathbb{E} [ \mu_{\omega} (\cdot) ] = \rho(\cdot)\,,$$
and, hence,
$$ \mathbb E [ \mu(\cdot) | \mathcal{F}_{0}^{\infty}] = (\mathbb P \times \rho)(\cdot)\,.$$
Therefore the probability measure $\mathbb{P} \times \rho$ is invariant for $(\Theta_t)_{t \geq 0}$ on $(\Omega \times \cX, \mathcal{F}_0^\infty \times \mathcal{B}(\cX))$. In words, the product measure with marginals $\mathbb{P}$ and $\rho$ is invariant for the random dynamical system restricted to one-sided path space. 
\subsection{Lyapunov spectrum}
\label{LyapSpec}

Consider a $C^k$ random dynamical system $(\theta, \varphi)$, i.e.~$\varphi(t, \omega, \cdot) \in C^k$ for all $t\in\mathbb T$ and $\omega\in\Omega$, where again $ \mathbb T \in \{ \R, \R_0^+\}$. Let's assume that $\cX$ is a smooth $m$-dimensional manifold and that $(\theta, \varphi)$ is  $C^1$. Recall that the \textit{linearization} or \textit{derivative} $\rmD \varphi(t,\omega,x)$ of $\varphi(t,\omega,\cdot)$ at $x \in \cX$ is a linear map from the tangent space $T_x$ to the tangent space $T_{\varphi(t,\omega,x)}$. If $\cX = \mathbb{R}^m$, the linearization is simply the Jacobian $m\times m$ matrix
$$ \rmD \varphi(t,\omega,x) = \frac{\partial \varphi(t, \omega,x)}{\partial x}\,.$$
Further assume that the random dynamical system possesses an invariant measure $\mu$. In case $\cX= \mathbb{R}^m$, this implies that $(\Theta,\rmD \varphi)$ is a random dynamical system with linear cocycle $\rmD \varphi$ over the metric dynamical system $(\Omega \times \cX, \mathcal F \times \mathcal{B}(\cX), (\Theta_t)_{t \in \mathbb T})$, see e.g. \cite[Proposition 4.2.1]{a98}. Generally, we have that $\rmD \varphi$ is a linear bundle random dynamical system on the tangent bundle $T\cX$ (see \cite[Definition 1.9.3, Proposition 4.25]{a98}).

In case the derivative can be written as a matrix, as for example for $\cX= \mathbb{R}^m$, the Jacobian $\rmD \varphi(t,\omega,x)$ satisfies \emph{Liouville's equation}
\begin{align} \label{Liouville}
\det \rmD \varphi(t, \omega,x) &= \exp \bigg( \int_{0}^{t} \trace \rmD f_0(\varphi(s,\omega) x) \rmd s \nonumber \\
&+ \sum_{j=1}^m \int_{0}^{t} \trace \rmD f_j(\varphi(s,\omega) x) \circ \rmd W_s^j \bigg).
\end{align}
We summarise the different versions of the \emph{Multiplicative Ergodic Theorem} for differentiable random dynamical systems in one-sided and two-sided time in the following theorem \cite[Theorem 3.4.1, Theorem 3.4.11, Theorem 4.2.6]{a98}, establishing a \emph{Lyapunov spectrum} with an associated filtration of random sets and, in two-sided time, with a splitting into invariant random subspaces.
\begin{thm} \label{MET}
\begin{enumerate} [a)]
\item Suppose the $C^1$-random dynamical system $(\theta,\varphi)$, where $\varphi$ is defined in forward time, has an ergodic invariant measure $\nu$ and satisfies the integrability condition
\begin{equation*}
\sup_{0 \leq t \leq 1} \ln^+ \| \rmD \varphi(t, \omega, x) \| \in L^1(\nu).
\end{equation*}
Then there exist a $\Theta$-invariant set $\Delta \subset \Omega \times \cX$ with $\nu (\Delta) = 1 $, a number $1 \leq p \leq m$ and real numbers $\lambda_1 > \dots > \lambda_p$, the \emph{Lyapunov exponents} with respect to $\nu$, such that for all $0 \neq v \in T_x \cX \cong \mathbb{R}^m$ and $(\omega,x) \in \Delta$
\begin{equation*}
\lambda(\omega, x, v) := \lim_{t \to \infty} \frac{1}{t} \ln \| \rmD \varphi(t, \omega, x)v \| \in \{\lambda_p, \dots, \lambda_1\}\,.
\end{equation*} .
Furthermore, the tangent space $T_x \cX \cong \mathbb{R}^m$ admits a filtration 
\begin{displaymath}
  \mathbb R^m = V_1(\omega, x) \supsetneq V_2(\omega,x)\supsetneq \dots \supsetneq V_p(\omega,x) \supsetneq V_{p+1}(\omega,x)= \{0\}\,,
\end{displaymath}
for all $(\omega,x) \in \Delta$ such that
\begin{equation*}
\lambda (\omega, x, v) = \lambda_{i}\quad \Longleftrightarrow \quad v \in V_i(\omega,x) \setminus V_{i+1}(\omega,x) \quad \text{ for all } i\in\{1, \dots, p\}\,.
\end{equation*}
In case the derivative can be written as a matrix, we have for all $(\omega,x) \in \Delta$
\begin{equation} \label{sumformula}
\lim_{t \to \infty} \frac{1}{t} \ln \det  \rmD \varphi(t, \omega, x) = \sum_{i=1}^p d_i \lambda_i\,,
\end{equation}
where $d_i$ is the multiplicity of the Lyapunov exponent $\lambda_i$ and $\sum_{i=1}^p d_i =m$.
\item If the cocycle $\varphi$ is defined in two-sided time and satisfies the above integrability condition also in backwards time, there exists the Oseledets splitting
$$ \mathbb{R}^m = E_1(\omega,x) \oplus \cdots \oplus E_p(\omega,x) $$
of the tangent space into random subspaces $E_i(\omega,x)$, the \emph{Oseledets spaces}, for all $(\omega,x) \in \Delta$. These have the following properties for all $(\omega,x) \in \Delta$:
\begin{enumerate} [(i)]
\item The Oseledets spaces are invariant under the derivative flow, i.e. for all $t \in \mathbb{R}$
$$ \rmD \varphi(t, \omega,x) E_i(\omega,x) = E_i( \Theta_t(\omega,x)) \,,$$
\item The Oseledets space $E_i$ corresponds with $\lambda_1$ in the sense that
\begin{equation*}
\lim_{t \to \infty} \frac{1}{t} \ln \| \rmD \varphi(t, \omega, x)v \|
= \lambda_{i}\quad \Longleftrightarrow \quad v \in E_i(\omega,x) \setminus \{0\} \quad \text{ for all } i\in\{1, \dots, p\}\,,
\end{equation*}
\item The dimension equals the multiplicity of the associated Lyapunov exponent, i.e.
$$ \dim E_i(\omega,x) = d_i\,.$$
\end{enumerate}
\end{enumerate}
\end{thm}

\subsection{Existence of random attractors} \label{appendix:randomattractor}
The existence of random attractors is proved via so-called absorbing sets. A set $B\in\mathcal D$ is called an \emph{absorbing set} if for almost all $\omega\in\Omega$ and any $D \in \mathcal D$, there exists a $T>0$ such that
\[
\varphi(t,\theta_{-t}\omega)D(\theta_{-t}\omega)\subset B(\omega)\fa t\geq T\,.
\]
A proof of the following theorem can be found in \cite[Theorem~3.5]{flandolischmalfuss96}.
\begin{thm}[Existence of random attractors]\label{ExistencePullback}
  Suppose that $(\theta,\varphi)$ is a continuous random dynamical system with an absorbing set $B$. Then there exists a unique random attractor $A$, given by
	\[
	A(\omega)
	:=
	\bigcap_{\tau\geq 0}\overline{\bigcup_{t\geq \tau} \varphi(t,\theta_{-t}\omega)B(\theta_{-t}\omega)}
	\faa \omega\in\Omega.
	\]
	Furthermore, $ \omega \mapsto A(\omega)$ is measurable with respect to $\mathcal{F}_{-\infty}^{0}$, i.e.~the past of the system.
\end{thm}
\begin{remark} \label{remark_mu_A}
Naturally, random attractors are related to invariant probability measures of a random dynamical system $(\theta, \varphi)$.
It follows directly from \cite[Proposition 4.5]{cf94} that, if the fibers of a random attractor $A$, i.e.~$\omega \mapsto A(\omega)$, are measurable with respect to $\mathcal F_{-\infty}^0$, there is an invariant measure $\mu$ for $(\theta, \varphi)$ such that $\omega \mapsto \mu_{\omega}$ is measurable with respect to $\mathcal F_{-\infty}^0$, i.e.~is a Markov measure, and satisfies $\mu_{\omega} (A(\omega)) =1$ for almost all $\omega \in \Omega$.
In particular, if there exists a unique invariant probability measure $\rho$ for the Markov semi-group $(P_t)_{t \ge 0}$, then the invariant Markov measure, supported on $A$, is unique by the one-to-one correspondence explained above.
Additionally, if the Markov semi-group is \emph{strongly mixing}, i.e.
$$ P_t f(x) \xrightarrow{t \to \infty} \int_\cX f(y) \rho (\rmd y) \fa \text{continuous and bounded } f: \cX \to \mathbb{R} \text{ and } x \in \cX\,,$$
then the set $\tilde A \in \mathcal{F} \times \mathbb{B}(\cX)$, given by $\tilde A(\omega) = \supp \mu_{\omega} \subset A(\omega)$ for almost all $\omega \in \Omega$, is a minimal weak random point attractor according to \cite[Proposition 2.20]{fgs16}.
\end{remark}

\subsection{Expectation of random isochron map}
\label{appendix:rel_rim}
We observe from equation~\eqref{eq:invariance2} that the random isochron map $\tilde \phi$ satisfies
\be \label{eq:expected_invariance2}
\mathbb{E}[\tilde \phi ( \varphi(t, \cdot, (\vartheta, r)), \theta_t \cdot, t)] = \mathbb{E}[\tilde \phi((\vartheta, r), \cdot,0)] +t\,.
\ee
%

Assume now that there is a function $\phi: \mathcal{R} \to \mathbb{R}$ such that for all $t$ in some interval $J=[0,T]$, $T >0$, we have
\begin{equation} \label{eq:intro_phi}
\mathbb{E}[ \phi ( \varphi(t, \omega, (\vartheta, r)))]= \mathbb{E}[\tilde \phi ( \varphi(t, \omega, (\vartheta, r)), \theta_t \omega, t)]\,.
\end{equation}
Then we obtain from equation~\eqref{eq:expected_invariance2} that
\be \label{eq:expected_invariance_phi}
\mathbb{E}[ \phi ( \varphi(t, \omega, (\vartheta, r)))] = \mathbb{E}[ \phi((\vartheta, r))] +t = \phi((\vartheta, r)) + t\,.
\ee
Hence, assuming the appropriate boundary conditions, we can deduce that $\phi = \bar \phi$, where $\bar \phi$ is the isochron function as derived above, satisfying equation~\eqref{eq:BKstat}. Furthermore, we can observe directly that 
\be \label{eq:phi_tildephi}
\bar \phi_{\textnormal{RDS}}((\vartheta, r)) = \mathbb{E}[\tilde \phi((\vartheta, r), \omega,0)] 
\ee
is the only cadidate for relation~\eqref{eq:intro_phi} to hold. When we insert equality~\eqref{eq:phi_tildephi} back into equation~\eqref{eq:intro_phi}, we obtain 
$$ \mathbb{E}[\mathbb{E}[ \tilde \phi ( \varphi(t, \omega, (\vartheta, r)), \omega', 0)]]= \mathbb{E}[\tilde \phi ( \varphi(t, \omega,(\vartheta, r) ), \theta_t \omega, t)].$$
If we choose $(\vartheta, r)$ to be a point on the random attractor, belonging to the CRPS $\psi$, say $(\vartheta, r)= \psi(0,\omega)$, then due to the fact that $ \tilde \phi ( \psi(t, \theta_t \omega), \theta_t \omega, t) = t$ for (almost) all $\omega \in \Omega$, this means that
\be \label{eq:doubleexpt}
\mathbb{E}[\mathbb{E}[ \tilde \phi ( \psi(t, \theta_t \omega), \omega', 0)]] = t.
\ee
Verifying equality~\eqref{eq:doubleexpt} would therefore lead to establishing $\bar \phi_{\textnormal{RDS}} = \bar \phi$.
We have not found a clear reasoning when and why (or why not) relation~\eqref{eq:doubleexpt} holds and leave it as an open problem to get a better understanding of this gap.

\bibliographystyle{abbrv}
\bibliography{mybibfile}

\end{document}